\documentclass[12pt,a4paper]{article}

\usepackage[utf8]{inputenc}
\usepackage[T1]{fontenc}
\usepackage{lmodern}
\usepackage[english]{babel}
\usepackage{amsmath,amssymb,amsthm,mathtools}
\usepackage{geometry}
\usepackage{enumitem}
\usepackage{graphicx}
\usepackage{tikz}
\usepackage{float}
\usetikzlibrary{calc}
\usepackage[hidelinks]{hyperref}
\hypersetup{
  pdftitle={Minimal Equiangular Hyperbolic Polyhedra in the Tetrahedral Range},
  pdfauthor={Andrey Egorov},
  pdfsubject={Hyperbolic geometry and convex polyhedra},
  pdfkeywords={hyperbolic polyhedra, equiangular polyhedra, volume, Andreev's theorem, minimal volume}
}

\newtheorem{theorem}{Theorem}[section]
\newtheorem{lemma}[theorem]{Lemma}
\newtheorem{proposition}[theorem]{Proposition}

\theoremstyle{definition}

\DeclareMathOperator{\Vol}{vol}
\DeclareMathOperator{\Area}{area}
\DeclareMathOperator{\arcosh}{arcosh}

\newcommand{\HH}{\mathbb H}
\newcommand{\RR}{\mathbb R}

\newcommand{\voct}{v_{\mathrm{oct}}}
\newcommand{\vtet}{v_{\mathrm{tet}}}
\newcommand{\Lob}{\Lambda}
\newcommand{\Gcat}{G}

\begin{document}

\title{Minimal Equiangular Hyperbolic Polyhedra\\
in the Tetrahedral Range}
\author{Andrey Egorov\\
Sobolev Institute of Mathematics, 4 Academician Koptyug Avenue,\\
630090 Novosibirsk, Russia\\
Corresponding author: \texttt{a.egorov2@g.nsu.ru}\\
ORCID: \href{https://orcid.org/0009-0007-8795-8148}{0009-0007-8795-8148}}
\date{}
\maketitle

\begin{abstract}
We study the minimum-volume problem for finite-volume convex hyperbolic polyhedra whose dihedral angles are all equal to a fixed number \(\alpha\). In the non-obtuse case one necessarily has
\[
\frac{\pi}{3}\le \alpha\le \frac{\pi}{2}.
\]
We prove that throughout the full range in which the regular hyperbolic tetrahedron with dihedral angle \(\alpha\) exists,
\[
\frac{\pi}{3}\le \alpha<\arccos\frac13,
\]
it is the unique minimum-volume equiangular hyperbolic polyhedron with prescribed angle \(\alpha\). The left endpoint is the known ideal case, while at the upper endpoint the tetrahedron shrinks to a point, with a Euclidean regular tetrahedron as its rescaled limit. The proof combines Andreev's theorem and the Schl\"afli formula with Atkinson's decomposition into atoroidal and prismatic parts, explicit volume estimates for ordinary prisms and complete orthoschemes, and a direct equiangular version of Inoue's edge surgery.
\end{abstract}



\section{Introduction}

Andreev's classical theorem gives a complete description of finite-volume convex hyperbolic polyhedra with dihedral angles not exceeding \(\pi/2\), in terms of their combinatorics and linear inequalities on the angles \cite{AndreevCompact,AndreevFinite,RHD}. In particular, once the combinatorial type of a polyhedron and a collection of non-obtuse dihedral angles are fixed, existence and uniqueness of a hyperbolic realization are reduced to checking the conditions of Andreev's theorem.

In this paper we consider a special but natural case: all dihedral angles of the polyhedron are equal to one number \(\alpha\). We call such a polyhedron \emph{equiangular}. It follows from Andreev's theorem that, for non-obtuse equiangular polyhedra, possible angles lie in the interval
\[
        \frac{\pi}{3}\le \alpha\le \frac{\pi}{2}.
\]
For \(\alpha=\pi/3\), all vertices are ideal and one obtains ideal \(\pi/3\)-equiangular Coxeter-type polyhedra. For \(\alpha=\pi/2\), one obtains right-angled polyhedra. The class of equiangular polyhedra is rich throughout this interval. In particular, all fullerenes are realized as equiangular hyperbolic polyhedra for every angle \(\alpha\in[\pi/3,\pi/2]\) \cite{EgorovFullerene}. Moreover, the regular tetrahedron exists for \(\pi/3\le\alpha<\arccos(1/3)\), and equiangular quadrilateral prisms, that is, equiangular parallelepipeds, exist for all \(\pi/3\le\alpha<\pi/2\); for \(\alpha=\pi/3\) one obtains an ideal equiangular parallelepiped, while as \(\alpha\to\pi/2\) a Euclidean degeneration occurs.

The endpoint cases of the minimal-volume problem are already known. For \(\alpha=\pi/3\), Atkinson's volume estimate for \(\pi/3\)-equiangular hyperbolic polyhedra implies that the minimum is attained by the ideal regular tetrahedron \cite[Theorem~2.6]{AtkinsonEquiangular}. Nonaka and Yoshida proved that the next \(\pi/3\)-equiangular polyhedron by volume is the regular ideal cube, that is, the ideal regular \(4\)-prism \cite[Theorem~3]{Nonaka2024}. For \(\alpha=\pi/2\), the minimum volume among all finite-volume right-angled hyperbolic polyhedra is attained by the triangular bipyramid \(P(3,2)\) \cite{EgorovP32}. This minimum volume is Catalan's constant. Thus a natural question is what happens between the two endpoint values of the angle.

Put
\[
        \alpha_0=\arccos\frac13\approx1.230959417 .
\]
The regular hyperbolic tetrahedron with dihedral angle \(\alpha\) exists precisely for
\[
        \frac{\pi}{3}\le\alpha<\alpha_0 .
\]
We refer to this interval as the \emph{tetrahedral range}. The left endpoint corresponds to the ideal regular tetrahedron, while as \(\alpha\to\alpha_0\) the tetrahedron shrinks to a point and becomes Euclidean after rescaling. Throughout this range the tetrahedron has the smallest possible number of faces, and its volume decreases from \(\vtet\) to zero.

Throughout the paper, \(\vtet\) and \(\voct\) denote the volumes of the
regular ideal hyperbolic tetrahedron and octahedron, respectively.

Prismatic \(4\)-circuits require a decomposition before right-angled volume estimates can be applied. We use Atkinson's decomposition to isolate atoroidal components and control their right-angled limits. The remaining graph-type case is reduced to an ordinary prism by a direct equiangular version of Inoue's edge surgery \cite[Section~8]{Inoue2008}. At every step only the angle of the lateral seam being removed is increased, directly from \(\alpha\) to \(\pi\), while all surviving angles remain equal to \(\alpha\).

The polar region affected by this deformation is the union of two spherical
triangles with side lengths \(L,L,s\), where \(L=\pi-\alpha\) and \(s\)
decreases from \(L\) to \(0\). Unlike the right-angled case, this region is
not a spherical bigon. We construct a natural \(1\)-Lipschitz collapse onto
the polar of the polyhedron obtained after deleting the seam. We retain the
overall scheme of Inoue's proof, but the two-triangle patch requires new
crossing estimates, a separate treatment of one-sided carriers, and a
modified final comparison of polar cellulations. The Schl\"afli formula then
shows that every such surgery strictly decreases volume.

The main result is the following.

\begin{theorem}[Main theorem]\label{thm:main}
Let
\[
        \frac{\pi}{3}\le\alpha<\arccos\frac13 .
\]
Then every finite-volume convex hyperbolic polyhedron \(P\) whose dihedral angles are all equal to \(\alpha\) satisfies
\[
        \Vol(P)\ge \Vol(T_\alpha),
\]
where \(T_\alpha\) is the regular tetrahedron with dihedral angle \(\alpha\). Equality holds only if \(P\) is isometric to \(T_\alpha\).
\end{theorem}

For \(\pi/3<\alpha<\arccos(1/3)\), the volume of \(T_\alpha\) is given by
\begin{equation}\label{eq:tetra-volume-intro}
        \Vol(T_\alpha)=
        3\int_\alpha^{\arccos(1/3)}
        \arcosh\left(
        \frac{\cos t}{1-2\cos t}
        \right)\,dt .
\end{equation}
For \(\alpha=\pi/3\), the right-hand side is understood as an improper integral and equals the volume of the ideal regular tetrahedron.

The paper is organized as follows. Section~2 collects the geometric tools and
Atkinson's decomposition. Sections~3 and~4 treat the atoroidal case and its
exceptional right-angled limit. Section~5 develops the polar-metric surgery
and reduces graph-type polyhedra to ordinary prisms. Section~6 proves the main
theorem.

\section{Preliminaries}

\subsection{Non-obtuse hyperbolic polyhedra}

By a hyperbolic polyhedron we mean a convex finite-volume polyhedron in the three-dimensional hyperbolic space \(\HH^3\). A vertex is called \emph{finite} if it lies in \(\HH^3\), and \emph{ideal} if it lies on the sphere at infinity \(\partial\HH^3\). A polyhedron is called \emph{non-obtuse} if all its dihedral angles are at most \(\pi/2\). We will often use the standard term \emph{right-angled} for polyhedra all of whose dihedral angles are equal to \(\pi/2\).

Let \(P\) be an abstract polyhedron, that is, a cell decomposition of the sphere realized by a convex Euclidean polyhedron. Let \(P^*\) be the dual graph. A simple cycle of length \(k\) in \(P^*\) is called a \emph{prismatic \(k\)-circuit} if no two of its edges lie in the same face of the dual graph. Equivalently, such a circuit intersects \(k\) pairwise disjoint edges of the original polyhedron. These circuits play a key role in Andreev's conditions.

We use the following form of Andreev's theorem \cite{AndreevCompact,AndreevFinite}; see also the corrected exposition in \cite{RHD}.

\begin{theorem}[{\cite[Theorem~3]{Atkinson2011}}]\label{thm:andreev}
Let \((P,\Theta)\) be a non-obtuse labelled abstract polyhedron with more than four vertices. Here \(\Theta(e)\in(0,\pi/2]\) is the prescribed dihedral angle at an edge \(e\). Then \((P,\Theta)\) is realized by a finite-volume convex hyperbolic polyhedron if and only if the following conditions hold.
\begin{enumerate}[label=\textup{(\arabic*)},itemsep=0.15em,topsep=0.3em]
    \item Exactly three or exactly four edges meet at each vertex.
    \item If edges \(e_i,e_j,e_k\) meet at a vertex, then
    \[
        \Theta(e_i)+\Theta(e_j)+\Theta(e_k)\ge \pi .
    \]
    \item If edges \(e_i,e_j,e_k,e_l\) meet at a vertex, then
    \[
        \Theta(e_i)+\Theta(e_j)+\Theta(e_k)+\Theta(e_l)=2\pi .
    \]
    \item If edges \(e_i,e_j,e_k\) form a prismatic \(3\)-circuit, then
    \[
        \Theta(e_i)+\Theta(e_j)+\Theta(e_k)<\pi .
    \]
    \item If edges \(e_i,e_j,e_k,e_l\) form a prismatic \(4\)-circuit, then
    \[
        \Theta(e_i)+\Theta(e_j)+\Theta(e_k)+\Theta(e_l)<2\pi .
    \]
    \item If \(P\) is a triangular prism and \(e_i,e_j,e_k,e_p,e_q,e_r\) are the six edges contained in the two triangular faces, then
    \[
        \Theta(e_i)+\Theta(e_j)+\Theta(e_k)+\Theta(e_p)+\Theta(e_q)+\Theta(e_r)<3\pi .
    \]
    \item If faces \(F_i\) and \(F_j\) meet along an edge \(e_{ij}\), faces \(F_j\) and \(F_k\) meet along an edge \(e_{jk}\), and \(F_i\) and \(F_k\) have exactly one common ideal vertex distinct from the endpoints of \(e_{ij}\) and \(e_{jk}\), then
    \[
        \Theta(e_{ij})+\Theta(e_{jk})<\pi .
    \]
\end{enumerate}
The realization, if it exists, is unique up to isometry of \(\HH^3\). The ideal vertices of the realization are precisely the four-valent vertices and the trivalent vertices for which equality holds in condition \textup{(2)}.
\end{theorem}

In the equiangular case \(\Theta(e)=\alpha\), for
\[
        \frac{\pi}{3}<\alpha<\frac{\pi}{2}
\]
Andreev's theorem immediately implies that all vertices are trivalent and finite. Indeed, a four-valent vertex would require \(4\alpha=2\pi\), and equality at a trivalent vertex would require \(3\alpha=\pi\). Moreover, prismatic \(3\)-circuits are impossible, since they would require \(3\alpha<\pi\). Prismatic \(4\)-circuits are allowed, because \(4\alpha<2\pi\).

\subsection{The Schl\"afli formula}

\begin{theorem}[{\cite[Theorem~9]{Atkinson2011}}]\label{thm:schlafli}
Let \(P_t\) be a smooth deformation of a finite-volume hyperbolic polyhedron on an interval where its combinatorics and its set of ideal vertices are fixed. Then
\begin{equation}\label{eq:schlafli}
        d\Vol(P_t)=-\frac12\sum_e l_e\,d\theta_e,
\end{equation}
where the sum is taken over all edges, \(\theta_e\) is the dihedral angle at \(e\), and \(l_e\) is the length of \(e\). If an edge is incident to an ideal vertex, its length is understood as the signed truncated length with respect to a chosen system of horospheres; the whole right-hand side is independent of this choice.
\end{theorem}

In particular, for compact polyhedra an angle-nondecreasing deformation is
volume-nonincreasing. We apply this observation on nondegenerate intervals
and justify separately any passage to a degenerate endpoint. This principle
is at the heart of Atkinson's lower bounds \cite{Atkinson2011}.

\subsection{The regular tetrahedron}

Let \(T_\alpha\) denote the regular hyperbolic tetrahedron with dihedral angle \(\alpha\). Let \(n_1,n_2,n_3,n_4\) be outward unit normals to its faces. The Gram matrix of these outward normals has the form
\[
        G_\alpha=(g_{ij}),\qquad
        g_{ii}=1,\qquad
        g_{ij}=-\cos\alpha\quad (i\ne j).
\]
Its eigenvalues are
\[
        1+\cos\alpha
        \quad\text{with multiplicity }3,
        \qquad
        1-3\cos\alpha
        \quad\text{with multiplicity }1.
\]
By the Gram criterion for hyperbolic simplices, the Gram matrix of outward normals of a non-degenerate hyperbolic simplex must have signature \((3,1)\); see, for instance, \cite[Ch.~II, Sec.~3]{VinbergGeometry}. Here the signature condition is equivalent to
\[
        1-3\cos\alpha<0,
        \qquad\text{that is}\qquad
        \alpha<\arccos\frac13 .
\]
The lower bound for \(\alpha\) follows from vertex links. The link of a vertex of a hyperbolic tetrahedron is a spherical, Euclidean, or hyperbolic triangle according as the vertex is finite, ideal, or hyperideal. In the regular case the angles of this link are
\[
        \alpha,\alpha,\alpha .
\]
For a finite vertex the sum of the angles of the link is greater than \(\pi\), and for an ideal vertex it is equal to \(\pi\). Within the signature range above, the finite-volume vertex condition is therefore precisely
\[
        3\alpha\ge \pi .
\]
If \(3\alpha>\pi\), all vertices are finite; if \(3\alpha=\pi\), all vertices are ideal. If \(3\alpha<\pi\), the vertices become hyperideal, and one no longer obtains an ordinary finite-volume tetrahedron in \(\HH^3\). Thus the regular finite-volume hyperbolic tetrahedron exists exactly for
\[
        \frac{\pi}{3}\le\alpha<\arccos\frac13 .
\]
For \(\alpha=\pi/3\) it is the ideal regular tetrahedron, for \(\pi/3<\alpha<\arccos(1/3)\) it is compact, and as \(\alpha\to\arccos(1/3)\) it shrinks to a point, with a Euclidean regular tetrahedron as its rescaled limit.

We will need a formula for the volume of \(T_\alpha\). From the inverse Gram matrix of the regular simplex, or equivalently from the standard formula for the edge length of a regular hyperbolic tetrahedron, we get
\[
        \cosh l(\alpha)=\frac{\cos\alpha}{1-2\cos\alpha}.
\]
Here \(l(\alpha)\) is the length of any edge of \(T_\alpha\). By the Schl\"afli formula \eqref{eq:schlafli},
\[
        \frac{d}{d\alpha}\Vol(T_\alpha)=-3l(\alpha),
\]
since a tetrahedron has six edges and all its dihedral angles are equal to \(\alpha\). Since \(\Vol(T_\alpha)\to0\) as \(\alpha\to\arccos(1/3)\), for \(\pi/3<\alpha<\arccos(1/3)\) we obtain
\begin{equation}\label{eq:tet-volume}
        \Vol(T_\alpha)=
        3\int_\alpha^{\arccos(1/3)}
        \arcosh\left(
        \frac{\cos t}{1-2\cos t}
        \right)\,dt .
\end{equation}
Here and below \(\Lambda\) denotes the Lobachevsky function
\[
        \Lambda(x)=-\int_0^x\log|2\sin t|\,dt.
\]
For \(\alpha=\pi/3\), the preceding volume formula is understood as an
improper integral and gives the volume of the ideal regular tetrahedron
\[
        \vtet=3\Lob(\pi/3)=1.014941\ldots .
\]

\subsection{The Klein model}

Consider Minkowski space \(\RR^{3,1}\) with scalar product
\[
        \langle x,y\rangle=-x_0y_0+x_1y_1+x_2y_2+x_3y_3.
\]
Hyperbolic space is realized as the upper sheet of the hyperboloid
\[
        \HH^3=\{x\in\RR^{3,1}\mid \langle x,x\rangle=-1,\, x_0>0\}.
\]
Central projection
\[
        (x_0,x_1,x_2,x_3)\longmapsto
        \left(\frac{x_1}{x_0},\frac{x_2}{x_0},\frac{x_3}{x_0}\right)
\]
identifies \(\HH^3\) with the unit ball
\[
        \mathbb B^3=\{u\in\RR^3\mid |u|<1\}.
\]
This is the Klein model. Geodesics are represented by chords of the ball, and hyperbolic planes by intersections of the ball with Euclidean affine planes.

If \(u=(u_1,u_2,u_3)\) and \(r^2=|u|^2\), then the hyperbolic volume element in the Klein model is
\begin{equation}\label{eq:klein-volume-element}
        dV_{\HH^3}=\frac{du_1\,du_2\,du_3}{(1-r^2)^2}.
\end{equation}
In particular, the hyperbolic volume of any region \(\Omega\subset\mathbb B^3\) is not smaller than its Euclidean volume in the Klein model.

For two hyperplanes given in the hyperboloid model by outward unit normals \(n_1,n_2\), their dihedral angle \(\theta\) is determined by
\[
        \langle n_1,n_2\rangle=-\cos\theta.
\]
This formula will be used to derive the parameters of the model prisms.

\subsection{Hyperideal vertices and polar truncation}

We will use the standard language of generalized hyperbolic polyhedra. In the Klein model, a hyperideal vertex is an intersection point of extensions of faces lying outside the closed ball \(\overline{\mathbb B^3}\), such that its polar plane intersects the ball. In the projective model, a point \(v\) outside the ball determines a polar plane \(v^*\). This plane is defined by Lorentzian polarity and has the following property: \(v^*\) is orthogonal to every hyperbolic plane passing through the hyperideal vertex \(v\). This description of polar truncation is used, for example, by Ushijima for generalized hyperbolic tetrahedra \cite[Section~2]{Ushijima2006}; see also the discussion of polar planes and truncations in \cite[Section~2]{VesninEgorovUpper}.

Polar truncation of a hyperideal vertex means intersecting the polyhedron with the half-space bounded by the plane \(v^*\). The new face is called a truncation face. Since the polar plane is orthogonal to all original faces passing through \(v\), all new edges arising from such a truncation have dihedral angle \(\pi/2\). This property will be used below for the polyhedron obtained from \(P(3,2)\) by replacing ideal vertices with quadrilateral faces.

\subsection{Orthoschemes and Kellerhals' formula}

A hyperbolic orthoscheme in \(\HH^3\) is a simplex all of whose dihedral angles are right angles except, possibly, for three angles at consecutive edges. These three angles are called essential. We will also use complete orthoschemes: they may have ideal or hyperideal vertices, and hyperideal vertices are truncated by polar planes.

Let a complete three-dimensional orthoscheme have essential angles
\[
        \beta_{01},\qquad \beta_{12},\qquad \beta_{23}.
\]
Put
\[
        \tan\theta=
        \frac{
        \sqrt{\cos^2\beta_{12}-\sin^2\beta_{01}\sin^2\beta_{23}}
        }{
        \cos\beta_{01}\cos\beta_{23}
        },
        \qquad
        0\le \theta<\frac{\pi}{2}.
\]
Then Kellerhals' formula for the volume of a complete orthoscheme is
\begin{align}\label{eq:kellerhals}
        \Vol(\mathcal O)=\frac14\bigg(&
        \Lambda(\beta_{01}+\theta)-\Lambda(\beta_{01}-\theta)
        +\Lambda\left(\frac{\pi}{2}+\beta_{12}-\theta\right)  \\
        &+\Lambda\left(\frac{\pi}{2}-\beta_{12}-\theta\right)
        +\Lambda(\beta_{23}+\theta)-\Lambda(\beta_{23}-\theta)
        +2\Lambda\left(\frac{\pi}{2}-\theta\right)
        \bigg).
\end{align}
This formula was obtained by Kellerhals
\cite[Equations~(4)--(5) and Theorem~2]{Kellerhals1989}; it applies, in
particular, to complete orthoschemes with one hyperideal vertex, which will
occur below.

\subsection{Spherical polars and the Hodgson--Rivin criterion}\label{subsec:rhs-prelim}

We recall the polar construction and the characterization of
Hodgson--Rivin \cite[Theorem~1.1]{RivinHodgson1993}; see also the earlier
announcement \cite[Theorem~5]{HodgsonRivinSmith}.

Let \(P\subset \HH^3\) be a compact convex hyperbolic polyhedron.  If \(e\) is an
edge of \(P\), we write \(\theta_e\) for its interior dihedral angle and
\[
        \theta_e^{\rm ext}=\pi-\theta_e
\]
for its exterior dihedral angle.  Let \(v\) be a trivalent vertex of \(P\), and
let the three edges incident to \(v\) have interior dihedral angles
\(\theta_1,\theta_2,\theta_3\).

There are two spherical triangles associated with \(v\).  The first one is the
ordinary spherical link \(L_v\), obtained by intersecting \(P\) with a small metric
sphere centered at \(v\).  The angles of \(L_v\) are precisely
\(\theta_1,\theta_2,\theta_3\).  The second one is the spherical polar of this
link.  It is this polar triangle, equivalently the Gauss image of the vertex,
which occurs in the theorem of Hodgson--Rivin.  Its side lengths are the
exterior dihedral angles
\begin{equation}\label{eq:polar-triangle-sides}
        \pi-\theta_1,\qquad \pi-\theta_2,\qquad \pi-\theta_3 .
\end{equation}
Thus an edge of \(P\) with interior dihedral angle \(\theta_e\) corresponds in the
polar metric to a spherical geodesic segment of length \(\pi-\theta_e\).  This is
why, below, increasing an interior angle from \(\pi/2\) to \(\pi\) corresponds to
shrinking a polar side from length \(\pi/2\) to length \(0\).

The \emph{spherical polar} \(P^\circ\) is obtained as follows.  For every vertex
\(v\) of \(P\) take the polar triangle described above.  If two vertices of \(P\)
are joined by an edge \(e\), glue the corresponding sides of their polar triangles;
these two sides have the same length \(\pi-\theta_e\).  The resulting object is a
piecewise spherical cone metric on \(S^2\), combinatorially dual to \(P\):
\[
\begin{array}{c|c}
P & P^\circ \\
\hline
\text{face }F & \text{cone point }F^\circ,\\
\text{edge }e & \text{geodesic edge }e^\circ,\\
\text{vertex }v & \text{spherical triangle }L_v^\circ.
\end{array}
\]
This dual cell decomposition is called the \emph{polar cellulation}.  Its
vertices are the cone points corresponding to primal faces, its edges are
dual to primal edges, and its two-cells are dual to primal vertices.  Since
the polyhedra used here are trivalent, their polar two-cells are triangles.
An auxiliary geodesic triangulation drawn on the same cone metric need not,
without an additional argument, be this actual polar cellulation.

It is important to distinguish side lengths of polar triangles from their angles.
Let \(F\) be a face of \(P\), and let \(x\) be a vertex of \(F\).  If
\(\phi_x(F)\) is the ordinary angle of the hyperbolic polygon \(F\) at \(x\), then
the contribution of the polar triangle \(L_x^\circ\) to the cone angle at
\(F^\circ\) is
\[
        \pi-\phi_x(F).
\]
Consequently, if \(F\) has vertices \(x_1,\ldots,x_m\), then the cone angle of the
polar metric at \(F^\circ\) is
\begin{equation}\label{eq:polar-cone-angle-face}
        \omega(F^\circ)
        =\sum_{i=1}^m\bigl(\pi-\phi_{x_i}(F)\bigr)
        =2\pi+\Area(F).
\end{equation}
The last equality is the Gauss--Bonnet formula for the hyperbolic polygon \(F\).
In particular, the polar of an actual compact hyperbolic polyhedron has all cone
angles strictly larger than \(2\pi\).

A closed geodesic in a piecewise spherical cone metric means a closed curve which
is locally geodesic, including at cone points.  We shall use the following
characterization.

\begin{theorem}[{Hodgson--Rivin \cite[Theorem~1.1]{RivinHodgson1993}}]\label{thm:rhs}
A piecewise spherical cone metric \(M\) on \(S^2\) is the spherical polar of a
compact convex hyperbolic polyhedron if and only if the following conditions hold:
\begin{enumerate}[label=\textup{(RH\arabic*)}]
\item \(M\) has curvature \(+1\) away from finitely many cone points;
\item every cone angle of \(M\) is strictly larger than \(2\pi\);
\item every nonconstant closed local geodesic in \(M\) has length strictly
larger than \(2\pi\).
\end{enumerate}
When these conditions hold, the corresponding hyperbolic polyhedron is unique up
to congruence in \(\HH^3\).
\end{theorem}

We shall apply the theorem only in the following way.  A family of cone metrics
\(M_t\) will be prescribed by gluing spherical polar triangles.  The side lengths
of these triangles are prescribed exterior dihedral angles.  Once
\textup{(RH1)}--\textup{(RH3)} are verified, Theorem~\ref{thm:rhs} produces a
unique compact convex hyperbolic polyhedron with those interior dihedral angles.
The argument is modeled on Inoue's proof of right-angled edge surgery
\cite[Section~8]{Inoue2008}.  The fixed sides of our polar triangles have
length \(\pi-\alpha\), rather than \(\pi/2\), so both the local crossing
argument and the final recovery of the polar cellulation require additional
work.

We record the spherical facts needed for that modification.  They also fix the
terminology used in the proof.  If \(x\) is a vertex of a spherical
triangulation, its \emph{closed star}, denoted by \(\operatorname{st}(x)\), is
the union of all closed triangles containing \(x\).  Thus the number of
triangles in \(\operatorname{st}(x)\) is the valence of \(x\); it is not
assumed to be four.  Its \emph{open star} is
\[
        \operatorname{ost}(x)
        =\operatorname{st}(x)\setminus\partial\operatorname{st}(x).
\]
Equivalently, it is the union of the relative interiors of all simplices
containing \(x\).

For the remainder of this subsection we assume
\[
        \frac{\pi}{3}<\alpha<\frac\pi2 .
\]
Put
\begin{equation}\label{eq:polar-L-beta}
        L=\pi-\alpha,
        \qquad c=\cos L,
        \qquad
        \cos\beta=\frac{c}{1+c}.
\end{equation}
In this open range,
\begin{equation}\label{eq:L-beta-ranges}
        \frac\pi2<L<\frac{2\pi}{3},
        \qquad -\frac12<c<0,
        \qquad \frac\pi2<\beta<\pi .
\end{equation}
The number \(\beta\) is the angle of the equilateral spherical triangle with
side length \(L\).

\subsubsection*{Spherical development}

We shall use the standard developing-map construction for a piecewise
spherical cone metric; compare the unfolding in
\cite[proof of Lemma~8.7]{CharneyDavis1993}.  Let \(\Sigma\) be the finite set
of cone points and put \(M_{\mathrm{reg}}=M\setminus\Sigma\).  Every point of
\(M_{\mathrm{reg}}\) has a neighborhood isometric to an open subset of the
unit sphere, and the transition maps between such spherical charts are
restrictions of isometries of \(S^2\).  Choose one chart and continue it along
paths by successively unfolding adjacent spherical triangles across their
common geodesic sides.  Continuation around a loop can return rotated by the
holonomy, so in general it is not single-valued on \(M_{\mathrm{reg}}\).
After passing to the universal cover, paths with the same endpoints are
homotopic relative to their endpoints and continuation is single-valued.  One
therefore obtains a local isometry
\[
        \operatorname{dev}\colon
        \widetilde{M_{\mathrm{reg}}}\longrightarrow S^2 .
\]
It is equivariant with respect to a holonomy representation
\(\pi_1(M_{\mathrm{reg}})\to\operatorname{Isom}(S^2)\), but it need not be
injective.  Since it is a local isometry, it preserves the length of every
developed path and sends each local geodesic to a parametrized arc of a great
circle.

In particular, the length of a developed local geodesic is the length of the
parametrized arc, not the spherical distance between its endpoints.  A
parametrized arc may traverse the same great circle more than once and may
therefore have length larger than \(2\pi\).  In the only application below no
covering-space ambiguity is present: \(K_s\) is a two-triangle topological
disc.  We place \(ABC\) in \(S^2\) and place \(ABD\) on the other side of the
great circle through \(AB\), equivalently by reflecting the second triangle
across that great circle.  This directly develops all of \(K_s\).

\begin{lemma}[The two special polar triangles]\label{lem:special-polar-triangles}
Let \(0<s\le L\), and let \(T(s)\) be the spherical triangle with side
lengths \(L,L,s\).  Denote by \(\delta(s)\) either angle adjacent to the side
of length \(s\), and by \(\gamma(s)\) the opposite angle.  Then
\begin{equation}\label{eq:delta-gamma}
 \cos\delta(s)=\cot L\tan\frac{s}{2},
 \qquad
 \cos\gamma(s)=\frac{\cos s-c^2}{1-c^2}.
\end{equation}
In particular,
\begin{equation}\label{eq:delta-gamma-bounds}
        \frac\pi2<\delta(s)\le\beta,
        \qquad 0<\gamma(s)\le\beta .
\end{equation}
At \(s=L\), the triangle is equilateral and
\(\delta(L)=\gamma(L)=\beta\); as \(s\to0\), one has
\(\delta(s)\to\pi/2\) and \(\gamma(s)\to0\).
\end{lemma}

\begin{proof}
Apply the spherical cosine rule first to a side of length \(L\) and then to
the side of length \(s\):
\[
 \cos L=\cos L\cos s+\sin L\sin s\cos\delta,
 \qquad
 \cos s=c^2+(1-c^2)\cos\gamma .
\]
These equations give \eqref{eq:delta-gamma}.  Since \(\cot L<0\) and
\(0<s\le L<\pi\), the first cosine is negative.  The remaining assertions
follow directly, using the equilateral cosine rule at \(s=L\).
\end{proof}

\begin{lemma}[Strict star-passage estimate]
\label{lem:ordinary-star-passage}
Let \(Q\) be the spherical polar of a compact equiangular hyperbolic
polyhedron with exterior dihedral angle \(L>\pi/2\), and let \(x\) be a
vertex of its polar cellulation. Let \(\lambda\) be a local geodesic, and let
\(\eta\) be the closure of a component of
\(\lambda\cap\operatorname{ost}(x)\).  Assume that the endpoints of \(\eta\)
lie on \(\partial\operatorname{st}(x)\). Then
\[
        \ell(\eta)>\pi .
\]
\end{lemma}

\begin{proof}
Inoue's star estimate gives \(\ell(\eta)\ge\pi\); see
\cite[Lemma~8.4]{Inoue2008}.  We only have to exclude equality in the present
equiangular range.

Develop the centre of the star to a point \(x\in S^2\). If \(uv\) is a
boundary side of the star, then it is opposite \(x\) in an equilateral
spherical triangle of side length \(L\). Hence
\[
        x\cdot u=x\cdot v=\cos L<0.
\]
Every point of the short side \(uv\) is the normalization of a positive
linear combination of \(u\) and \(v\). Therefore every developed boundary
point \(p\) satisfies
\[
        x\cdot p<0,
        \qquad d_{S^2}(x,p)>\frac\pi2.
\]
If \(\eta\) passes through \(x\), its two parts from \(x\) to the boundary
both have length greater than \(\pi/2\). If it avoids \(x\), develop the
punctured star and \(\eta\) to \(S^2\). A great-circle arc of length exactly
\(\pi\) has antipodal endpoints, whereas both developed endpoints have
negative scalar product with the same vector \(x\). Thus equality is
impossible in either case.
\end{proof}

The lemma will be applied below only inside the known polar metric
\(P_1^\circ\), never directly to a candidate metric whose realizability has
not yet been proved.

We shall call a geodesic segment \emph{transverse} to a spherical
triangulation if it avoids the vertices and meets every edge that it crosses
in the interior of that edge and with distinct tangent directions.

We also need a local limiting convention.  A \emph{one-sided limit} of a
transverse passage through a fixed finite union of triangles is obtained by
moving its developed supporting great circle towards an edge or a vertex,
always from the same chosen side of the triangulation.  The developed
oriented arcs, together with their marked intersections with the relevant
sides, are required to converge.  Thus a one-sided limit is a portion of the
original nontransverse geodesic viewed in a chosen local carrier; it is not a
new globally defined geodesic and no closed perturbation is being asserted.

\begin{lemma}[Crossing the reflex two-triangle patch]
\label{lem:reflex-patch-crossing}
Let \(0<s\le L\), and let
\[
        K_s=ABC\cup_{AB}ABD
\]
be the union of two spherical triangles with
\[
 |AB|=s,\qquad |AC|=|BC|=|AD|=|BD|=L,
\]
developed on opposite sides of \(AB\).  If a transverse local geodesic segment
\(\kappa\) crosses, in this order, the sides
\[
        AC,\quad AB,\quad AD,
\]
then
\[
        \ell(\kappa)>\pi.
\]
The symmetric assertion with \(BC,BA,BD\) in place of \(AC,AB,AD\) also
holds.

The same strict estimate holds for every nonconstant one-sided limit of either
of these transverse passages.  In particular, the marked intersections with
\(AC,AB,AD\), or with \(BC,BA,BD\), may converge to endpoints of those sides,
and the limiting passage may run along a boundary edge of its carrier.

There is also a vertex version. Suppose that a local geodesic in \(K_s\)
passes through \(A\), with one germ in \(ABC\) and the other in \(ABD\), and
that the link angle between these germs on the \(K_s\)-side is at least
\(\pi\). Its
subarc from \(BC\) to \(BD\) through \(A\) has length \(>\pi\).
Symmetrically, a local geodesic through \(B\) from \(AC\) to \(AD\) has
length \(>\pi\).
\end{lemma}

\begin{proof}
Put
\[
 q=\sin\frac{s}{2},\qquad r=\cos\frac{s}{2},\qquad
 z=\sqrt{1-c^2-q^2},
\]
where \(c=\cos L<0\).  Since
\[
 q^2\le\sin^2\frac L2=\frac{1-c}{2},
\]
we have
\[
 z^2\ge\frac{(1-c)(1+2c)}2>0.
\]
After a rotation of \(S^2\), the developed vertices are
\[
 \begin{aligned}
 A&=(r,-q,0),& B&=(r,q,0),\\
 C&=(c/r,0,z/r),&D&=(c/r,0,-z/r).
 \end{aligned}
\]
Indeed, these vectors are unit vectors,
\[
 A\cdot B=\cos s,
 \qquad
 A\cdot C=B\cdot C=A\cdot D=B\cdot D=c,
\]
and \(D\) is the reflection of \(C\) across the plane of the great circle
containing \(AB\).  Thus the coordinates realize exactly the two prescribed
triangles.

They also give the positive linear relation
\begin{equation}\label{eq:reflex-patch-linear-relation}
        A+B+\mu(C+D)=0,
        \qquad
        \mu=-\frac{r^2}{c}>0 .
\end{equation}

Let \(P\in AC\), \(Q\in AB\), and \(R\in AD\) be the three crossing points.
The developed segment lies on a great circle
\(G=\nu^\perp\cap S^2\).  Choosing the sign of its normal \(\nu\), write
\[
 \nu\cdot A=-a<0,\qquad
 \nu\cdot B=b>0,\qquad
 \nu\cdot C=u>0,\qquad
 \nu\cdot D=v>0.
\]
The signs follow from the three transverse crossings.  Unnormalized vectors
on the rays through the crossing points are
\[
 p_0=uA+aC,\qquad
 q_0=bA+aB,\qquad
 r_0=vA+aD.
\]
All coefficients are positive, so their normalizations are precisely the
points \(P,Q,R\) on the short spherical sides.

Taking the scalar product of
\eqref{eq:reflex-patch-linear-relation} with \(\nu\) gives
\[
        b-a=-\mu(u+v).
\]
Using \(B=-A-\mu(C+D)\), we therefore obtain
\[
 \begin{aligned}
 q_0
 &=bA+aB\\
 &=(b-a)A-a\mu(C+D)\\
 &=-\mu\bigl((u+v)A+a(C+D)\bigr)\\
 &=-\mu(p_0+r_0).
 \end{aligned}
\]
The points \(P\) and \(R\) cannot be antipodal.  Otherwise their common line
would belong to both planes \(\operatorname{span}(A,C)\) and
\(\operatorname{span}(A,D)\), whose intersection is
\(\operatorname{span}(A)\) because the displayed coordinates make
\(A,C,D\) linearly independent.  This would force \(P=\pm A\), contrary to
transversality.

Consequently the antipodal point \(-Q\) lies in the interior of the short
spherical arc from \(P\) to \(R\): its direction is a positive linear
combination of the directions of \(P\) and \(R\).  Hence \(Q\) lies on the
complementary long arc from \(P\) to \(R\).  Since the developed segment
\(\kappa\) passes through \(Q\), it is this long arc, and
\[
        \ell(\kappa)=2\pi-d_{S^2}(P,R)>\pi.
\]
The case on the \(B\)-side follows by interchanging \(A\) and \(B\).

We next justify the assertion about one-sided limits.  Let
\(\kappa_n\) be transverse passages of the type just considered, with marked
points \(P_n\in AC\), \(Q_n\in AB\), and \(R_n\in AD\), and suppose that they
converge from the chosen side to a nonconstant limiting passage \(\kappa\).
The transverse part of the proof shows that \(\kappa_n\) is the long arc from
\(P_n\) to \(R_n\); hence
\[
        \ell(\kappa_n)=2\pi-d_{S^2}(P_n,R_n).
\]
After passing to the limit, write \(P_n\to P\in AC\) and
\(R_n\to R\in AD\).  The points \(P\) and \(R\) cannot be distinct and
antipodal.  Indeed, their common line would then lie in both
\(\operatorname{span}(A,C)\) and \(\operatorname{span}(A,D)\).  As above, the
intersection of these two planes is \(\operatorname{span}(A)\), so one of the
two points would be \(-A\); but \(-A\) lies on neither of the short sides
\(AC,AD\).  Therefore \(d_{S^2}(P,R)<\pi\).  The choice of one side fixes the
long, rather than the short, limiting arc, and continuity gives
\[
        \ell(\kappa)=2\pi-d_{S^2}(P,R)>\pi.
\]
If \(P=R\), the limiting oriented arc is a full great circle and has length
\(2\pi\), so the same conclusion holds.  This argument also covers a marked
point arriving at a vertex or a limiting great circle coinciding with a
boundary side: only the endpoints and the chosen long arc enter the length
calculation.  The \(B\)-side limits are identical by symmetry.

It remains to prove the vertex version. Let \(P\in BC\), \(R\in BD\), and
put
\[
        x=d(A,P),\qquad y=d(A,R).
\]
Let \(u=\angle PAB\) and \(v=\angle BAR\), measured inside the two special
triangles.  Then
\[
        0\le u,v\le\delta<\pi,
        \qquad
        \pi\le u+v<2\pi,
\]
where the lower bound for the sum is local geodesicity at \(A\).

If \(s\ge\pi/2\), every point of \(BC\) and \(BD\) is at distance at least
\(\pi/2\) from \(A\). Indeed, its representing vector is the normalization
of a nonnegative linear combination of the endpoint vectors, and their scalar
products with \(A\) are \(\cos s\le0\) and \(c<0\). Equality can occur only
when \(s=\pi/2\) and the point is \(B\). It cannot occur for both \(P\) and
\(R\), since then \(u=v=0\), contrary to \(u+v\ge\pi\). Hence
\(x+y>\pi\).

Suppose that \(s<\pi/2\). In the tangent plane at \(A\), take the direction
of \(AB\) as angle zero. Intersecting the great-circle ray which makes angle
\(u\) with \(AB\) with the great circle \(BC\), using the displayed
coordinates, gives
\[
        \cot x=\cot s\cos u+\frac{c}{2rz}\sin u.
\]
For completeness, write the ray as
\(\cos x\,A+\sin x\,\tau(u)\), take its scalar product with \(B\times C\),
and solve the resulting linear equation for \(\cot x\).
The identical calculation in \(ABD\) gives
\[
        \cot y=\cot s\cos v+\frac{c}{2rz}\sin v.
\]
Here \(\cot s>0\), \(c/(2rz)<0\), and
\[
 \cos u+\cos v
 =2\cos\frac{u+v}{2}\cos\frac{u-v}{2}\le0,
 \qquad
 \sin u+\sin v>0.
\]
Indeed, \(|u-v|<\pi\), so the second cosine in the product is positive.
Consequently \(\cot x+\cot y<0\). Since \(0<x,y<\pi\),
\[
 \sin(x+y)=\sin x\sin y\,(\cot x+\cot y)<0,
\]
and therefore \(x+y>\pi\). The assertion through \(B\) follows by symmetry.
\end{proof}

\subsection{Atkinson's decomposition}

Let \(P\) be a non-obtuse hyperbolic polyhedron. We use the decomposition
of \cite[Section~4]{Atkinson2011}, which allows arbitrary non-obtuse angles.
Atkinson first considers the double \(Q_P\) of \(P\), which is a hyperbolic
cone-manifold whose cone angle along an edge \(e\) is twice the dihedral angle
of \(P\) at \(e\).  He then forms the topological right-angled orbifold
\(Q_P^\perp\) by replacing all cone angles by \(\pi\).  We write \(P^\perp\)
for its quotient by the reflection which interchanges the two copies of
\(P\).  This notation is topological: \(P^\perp\) need not itself be realized
as a compact right-angled hyperbolic polyhedron.

In \(Q_P^\perp\), prismatic \(4\)-circuits correspond to Euclidean
two-dimensional suborbifolds of type \(S^2(2,2,2,2)\).  Choose a reduced
Bonahon--Siebenmann decomposition along such
suborbifolds. Its atoroidal components are canonical, while adjacent prism
regions separated by a layer-preserving gluing are amalgamated; see
\cite[Sections~4.2 and 6.4]{Atkinson2011}. The decomposition descends to
\(P^\perp\), and hence to the original polyhedron, along topological
quadrilaterals. A polyhedron is said to be of \emph{graph type} when no
atoroidal component occurs. In that case we use the reduced prism graph
\(G(P)\) of Atkinson--Rafalski, rather than an arbitrary visible subdivision
of the boundary of \(P\) into prisms.

\begin{proposition}[{\cite[Proposition~2]{Atkinson2011}}]\label{prop:atkinson-realization}
Let \(Q\) be a polyhedral component corresponding to an atoroidal component of the Bonahon--Siebenmann orbifold decomposition of a non-obtuse polyhedron \(P\). Let \(R\) be the abstract polyhedron obtained from \(Q\) by adding one triangular or quadrilateral face for each boundary component created by the cutting. The old edges are assigned the original dihedral angles, while all newly introduced edges are assigned angle \(\pi/2\). Then \(R\) is realized as a hyperbolic polyhedron.
\end{proposition}

\begin{proposition}[{\cite[Proposition~3]{Atkinson2011}}]\label{prop:atkinson-volume}
In the notation of Proposition~\ref{prop:atkinson-realization}, the rehyperbolization procedure does not increase the volume of the atoroidal component:
\[
        \Vol(R)\le \Vol(Q).
\]
\end{proposition}

\begin{theorem}[Right-angled deformation; {\cite[Section~8, pp.~207--208]{Atkinson2011}}]\label{thm:atkinson-deformation}
Suppose that \(P_\alpha\) is a compact non-tetrahedral equiangular polyhedron, with
\(\pi/3<\alpha<\pi/2\), and that \(R\) is an atoroidal polyhedron obtained by
the rehyperbolization of one of its atoroidal components.  Then there exists
an angle-nondecreasing deformation, with possible face degenerations, taking
\(R\) to a finite-volume right-angled hyperbolic polyhedron \(R^\perp\).  The
volume does not increase:
\[
        \Vol(R)\ge \Vol(R^\perp).
\]
\end{theorem}

We use the componentwise deformation described in
\cite[Section~8]{Atkinson2011}, following the rehyperbolization of
Section~4.2 and the face-degeneration construction of Section~5
(Corollary~6 and the subsequent volume-continuity argument).

\begin{theorem}[{\cite[Theorem~10]{Atkinson2011}}]\label{thm:atkinson-ra-bound}
Let \(Q\) be a finite-volume right-angled hyperbolic polyhedron with \(N_\infty\) ideal and \(N_F\) finite vertices. Then
\[
        \Vol(Q)\ge
        \frac{4N_\infty+N_F-8}{32}\,\voct,
\]
where \(\voct\) is the volume of the ideal regular octahedron.
\end{theorem}

\begin{theorem}[Graph-type structure; {\cite[Sections~4.2 and 6.4]{Atkinson2011};
\cite[Section~3, proof of Lemma~3.5]{AtkinsonRafalski}}]\label{thm:graph-type-structure}
Let \(P\) be a non-obtuse polyhedron of graph type.  The Seifert-fibered
components of a reduced geometric decomposition of \(Q_P^\perp\) have
singular loci of prism type.  In the quotient they give topological prism
blocks glued in pairs along quadrilateral cutting faces.  Associate one vertex
to every such prism component and one edge to every gluing quadrilateral.  For
a polyhedral orbifold the resulting connected gluing graph \(G(P)\) is a tree.
If \(G(P)\) has one vertex, \(P\) has the combinatorial type of an ordinary
\(n\)-prism.  If it has more than one vertex, it has at least two leaves; every
leaf in the reduced decomposition has prism degree \(n\ge5\), and exactly
\(n-3\) of its quadrilateral side faces remain free in the boundary of \(P\),
where they form a linear chain.
\end{theorem}

For an equiangular \(P\) with \(\pi/3<\alpha<\pi/2\), every prism block
here has degree at least four. Indeed, the cutting
quadrilaterals are mutually disjoint, so the attaching side faces of a fixed
prism block are pairwise nonadjacent in their cyclic order.  A triangular
prism block therefore has valence at most one in \(G(P)\).  It cannot be the
only block, since then \(P\) itself would be a triangular prism and its
equatorial prismatic \(3\)-circuit would contradict Andreev's inequality
\(3\alpha<\pi\).  Nor can it be a leaf of a nontrivial reduced tree, since
the theorem gives \(n\ge5\) for every such leaf.  Thus \(n\ge4\) for every
block; the stronger lower bound \(n\ge5\) concerns leaves.

The blocks in this statement are components after cutting along internal
quadrilaterals.  They are not asserted to be convex hyperbolic subpolyhedra of
\(P\).  Indeed, Atkinson's prism regions can have non-geodesic cutting faces
and \emph{virtual edges}; see \cite[Section~6.4]{Atkinson2011}.  What is used
below is their topological prism product structure and the actual free
quadrilateral faces and seams which remain on the boundary of \(P\).

\subsection{The finite-volume right-angled polyhedron of minimal volume}

Let \(P(3,2)\) denote the triangular bipyramid with three ideal and two finite vertices; see Figure~\ref{fig:p32}. It was proved in \cite{EgorovP32} that it has minimal volume among all finite-volume right-angled hyperbolic polyhedra.

\begin{figure}[ht]
\begin{center}
\begin{tikzpicture}[scale=1]
    \coordinate (A) at (0,0,2);
    \coordinate (B) at (-2,0,0);
    \coordinate (C) at (2,0,0);
    \coordinate (V1) at (0,2.5,0.5);
    \coordinate (V2) at (0,-2.5,0.5);
    \coordinate (AB) at (-1,0,1);
    \coordinate (BC) at (0.25,0,0);
    \coordinate (CA) at (1,0,1);
    \coordinate (O) at (0,0,0.5);
    \draw[very thick, black] (A) -- (B);
    \draw[very thick, black, dashed] (C) -- (B);
    \draw[very thick, black] (C) -- (A);
    \draw[very thick, black] (A) -- (V1);
    \draw[very thick, black] (B) -- (V1);
    \draw[very thick, black] (C) -- (V1);
    \draw[very thick, black] (A) -- (V2);
    \draw[very thick, black] (B) -- (V2);
    \draw[very thick, black] (C) -- (V2);
    \draw[thick, gray] (AB) -- (V1);
    \draw[thick, gray, dotted] (BC) -- (V1);
    \draw[thick, gray] (CA) -- (V1);
    \draw[thick, gray] (AB) -- (V2);
    \draw[thick, gray, dotted] (BC) -- (V2);
    \draw[thick, gray] (CA) -- (V2);
    \draw[thick, gray, dashed] (V1) -- (V2);
    \draw[thick, gray, dotted] (A) -- (O);
    \draw[thick, gray, dotted] (B) -- (O);
    \draw[thick, gray, dotted] (C) -- (O);
    \draw[thick, gray, dotted] (AB) -- (O);
    \draw[thick, gray, dotted] (BC) -- (O);
    \draw[thick, gray, dotted] (CA) -- (O);
    \foreach \v in {V1,V2} {\fill[black] (\v) circle (2pt);}
    \foreach \v in {A,B,C} {
        \fill[white] (\v) circle (4pt);
        \draw[red] (\v) circle (4pt);
        \fill[red] (\v) circle (2pt);
    }
    \node[] at (-0.2,-3) {$(a)$};
\end{tikzpicture}
\qquad
\begin{tikzpicture}[scale=1.2]
    \coordinate (A) at (0,0);
    \coordinate (B) at (3,0);
    \coordinate (C) at (1.5,2.5);
    \coordinate (D) at (1.5,0.75);
    \coordinate (E) at (1.5,1.5);
    \draw[very thick, black] (A) -- (B) -- (C) -- cycle;
    \draw[very thick, black] (A) -- (D) -- (B) -- (E) -- cycle;
    \draw[very thick, black] (C) -- (D);
    \draw[very thick, black] (D) -- (E);
    \foreach \v in {C,D} {\fill[black] (\v) circle (2pt);}
    \foreach \v in {A,B,E} {
        \fill[white] (\v) circle (4pt);
        \draw[red] (\v) circle (4pt);
        \fill[red] (\v) circle (2pt);
    }
    \node[] at (1.47,-1.5) {$(b)$};
\end{tikzpicture}
\end{center}
\caption{The triangular bipyramid $P(3,2)$ and its Schlegel diagram. Ideal
vertices are marked by open circles.}
\label{fig:p32}
\end{figure}
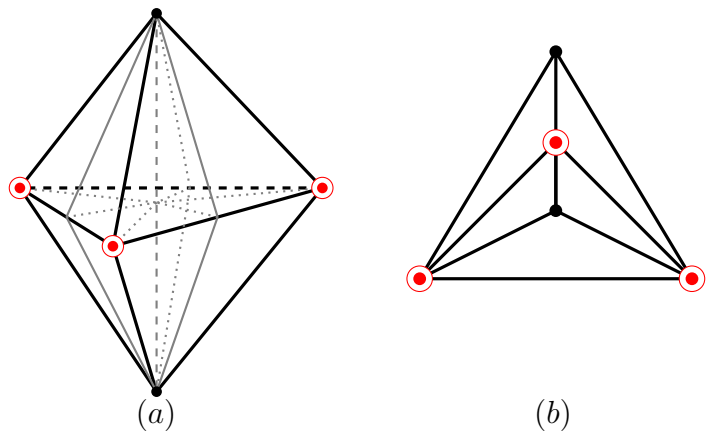

Moreover,
\[
        \Vol(P(3,2))=2\Lambda\left(\frac{\pi}{4}\right).
\]
This number is Catalan's constant
\[
        G=\sum_{k=0}^{\infty}\frac{(-1)^k}{(2k+1)^2}=0.915965\ldots .
\]
Since
\[
        \voct=8\Lambda\left(\frac{\pi}{4}\right),
\]
we have
\begin{equation}\label{eq:p32-volume}
        \Vol(P(3,2))=G=\frac{\voct}{4}.
\end{equation}
\section{The atoroidal case}

In this section we estimate polyhedra for which Atkinson's decomposition has an atoroidal component. We first record a small gap in the volumes of finite-volume right-angled polyhedra.

\begin{lemma}[Consequence of the right-angled minimum theorem]
\label{lem:ra-gap}
Let \(Q\) be a finite-volume right-angled hyperbolic polyhedron. If \(Q\) is not isometric to \(P(3,2)\), then
\[
        \Vol(Q)>\vtet.
\]
\end{lemma}

\begin{proof}
Put
\[
        w=4N_\infty+N_F-8.
\]
At a finite vertex of a finite-volume right-angled polyhedron three edges
meet, whereas an ideal vertex is four-valent.  Hence
\[
        3N_F+4N_\infty=2E.
\]
In particular, \(N_F\), and therefore \(w\), is even.  Atkinson's
bound \cite[Theorem~10]{Atkinson2011}, stated above as
Theorem~\ref{thm:atkinson-ra-bound}, gives
\[
        \Vol(Q)\ge
        \frac{w}{32}\,\voct .
\]
If \(w\ge10\), then
\[
        \Vol(Q)\ge \frac{10}{32}\voct
        =1.144956\ldots>1.014941\ldots=\vtet .
\]
There is no intermediate case \(w=9\): if \(w<10\), then \(w\le8\).

The cases with \(w\le8\) are exactly the small cases analyzed in
\cite[Lemmas~3.1--3.4]{EgorovP32}.  The elementary exclusions there first
rule out \(N_\infty=0,1\) and the pairs with too few total vertices.  The
remaining pairs are
\[
 (2,4),\ (2,6),\ (2,8),\ (3,2),\ (3,4).
\]
The first two are impossible, \((3,2)\) gives exactly \(P(3,2)\), and the
other two realizable cases are \(P(2,8)\), of volume \(2\Gcat\), and
\(P(3,4)\), of volume \(1.505361\ldots\).  Both volumes are larger than
\(\vtet\).  Therefore the only finite-volume right-angled polyhedron whose
volume does not exceed \(\vtet\) is \(P(3,2)\).
\end{proof}

\begin{lemma}[Estimate in the atoroidal case]\label{lem:atoroidal-general}
Let \(P_\alpha\) be an equiangular hyperbolic polyhedron with
\[
        \frac{\pi}{3}<\alpha<\arccos\frac13 .
\]
Assume that Atkinson's decomposition has an atoroidal component whose right-angled limit is not isometric to \(P(3,2)\). Then
\[
        \Vol(P_\alpha)>\Vol(T_\alpha).
\]
\end{lemma}

\begin{proof}
Let \(Q\) be such an atoroidal component, let \(R\) be its
rehyperbolization, and let \(R^\perp\) be the right-angled limit supplied by
Theorem~\ref{thm:atkinson-deformation}.  Propositions~\ref{prop:atkinson-realization}
and~\ref{prop:atkinson-volume} give
\[
        \Vol(P_\alpha)\ge \Vol(Q)\ge \Vol(R)\ge \Vol(R^\perp).
\]
By assumption, \(R^\perp\) is not isometric to \(P(3,2)\). Hence, by Lemma~\ref{lem:ra-gap},
\[
        \Vol(R^\perp)>\vtet.
\]
Since \(\Vol(T_\alpha)<\vtet\) for all \(\alpha>\pi/3\), we obtain
\[
        \Vol(P_\alpha)>\Vol(T_\alpha).
\]
\end{proof}

\section{The exceptional limit \texorpdfstring{\(P(3,2)\)}{P(3,2)}}

Let \(Q\) be an atoroidal component whose right-angled limit is \(P(3,2)\),
and let \(R\) be its rehyperbolization from
Proposition~\ref{prop:atkinson-realization}. The polyhedron \(R\) is compact
and trivalent, with all angles at least \(\alpha>\pi/3\), so it has no
prismatic \(3\)-circuit. It is not a tetrahedron: either it has a
quadrilateral cutting face or it is the original non-tetrahedral polyhedron.
Thus it has no triangular faces, since in its dual triangulation the
neighbours of a degree-three vertex would form a nonfacial \(3\)-cycle
unless the triangulation were tetrahedral.

In Atkinson's face-degeneration construction
\cite[Sections~5 and~8]{Atkinson2011}, quadrilateral faces are replaced by
four-valent ideal vertices, while the remaining vertices are finite.
Reversing the three such degenerations in \(P(3,2)\) gives the combinatorics
of \(R\) in Figure~\ref{fig:p32-opened}, with three distinguished
quadrilateral faces \(F_1,F_2,F_3\). A face \(F_i\) may be either an old
face or a cutting face. Every introduced edge lies on one of these faces;
all edges outside their boundaries have angle \(\alpha\).

Increase to \(\pi/2\) the angles on
\(\partial F_1\cup\partial F_2\cup\partial F_3\), keeping all other
angles equal to \(\alpha\), and denote the resulting realization by
\(\widehat B_\alpha\). To verify existence along this deformation, note
that every prismatic \(4\)-circuit is parallel to an \(F_i\) and crosses
the four unchanged edges leaving it, with angle sum \(4\alpha<2\pi\).
Vertex sums remain greater than \(\pi\); there are no prismatic
\(3\)-circuits, ideal vertices, or triangular-prism exception. Andreev's
theorem therefore gives compact realizations throughout the deformation.
The Schl\"afli formula and Proposition~\ref{prop:atkinson-volume} yield
\begin{equation}\label{eq:component-to-Bhat}
        \Vol(Q)\ge \Vol(R)\ge \Vol(\widehat B_\alpha).
\end{equation}
Now all three distinguished quadrilaterals have the same labels.  The
combinatorial symmetries preserving the two classes of labelled edges are
realized by isometries, by uniqueness in Andreev's theorem; in particular
\(\widehat B_\alpha\) has the order-\(12\) symmetry used below.

\begin{figure}[ht]
\centering
\begin{tikzpicture}[scale=0.72, line join=round, line cap=round,
    v/.style={circle, fill=black, inner sep=1.6pt},
    old/.style={very thick, black},
    cut/.style={very thick, red},
    every node/.style={font=\small}]
   	\draw[draw=black, thin, solid] (0.00,3.00) circle (0.1);
   \draw[draw=black, thin, solid] (0.00,2.00) circle (0.1);
   \draw[draw=black, thin, solid] (0.00,0.00) circle (0.1);
   \draw[draw=black, thin, solid] (-1.00,1.00) circle (0.1);
   \draw[draw=black, thin, solid] (1.00,1.00) circle (0.1);
   \draw[draw=black, thin, solid] (0.00,-1.00) circle (0.1);
   \draw[draw=black, thin, solid] (0.00,3.00) -- (0.00,2.00);
   \draw[draw=red!80!black, thick] (0.00,2.00) -- (-1.00,1.00);
   \draw[draw=red!80!black, thick] (-1.00,1.00) -- (0.00,0.00);
   \draw[draw=red!80!black, thick] (0.00,0.00) -- (1.00,1.00);
   \draw[draw=red!80!black, thick] (1.00,1.00) -- (0.00,2.00);
   \draw[draw=black, thin, solid] (0.00,0.00) -- (0.00,-1.00);
   \draw[draw=black, thin, solid] (-3.00,-2.00) circle (0.1);
   \draw[draw=black, thin, solid] (-4.00,-1.00) circle (0.1);
   \draw[draw=black, thin, solid] (-6.00,-2.00) circle (0.1);
   \draw[draw=black, thin, solid] (-4.50,-3.50) circle (0.1);
   \draw[draw=black, thin, solid] (-6.00,-2.00) -- (0.00,3.00);
   \draw[draw=red!80!black, thick] (-6.00,-2.00) -- (-4.00,-1.00);
   \draw[draw=red!80!black, thick] (-4.00,-1.00) -- (-3.00,-2.00);
   \draw[draw=red!80!black, thick] (-3.00,-2.00) -- (-4.50,-3.50);
   \draw[draw=red!80!black, thick] (-4.50,-3.50) -- (-6.00,-2.00);
   \draw[draw=black, thin, solid] (3.00,-2.00) circle (0.1);
   \draw[draw=black, thin, solid] (4.00,-1.00) circle (0.1);
   \draw[draw=black, thin, solid] (6.00,-2.00) circle (0.1);
   \draw[draw=black, thin, solid] (4.50,-3.50) circle (0.1);
   \draw[draw=red!80!black, thick] (3.00,-2.00) -- (4.50,-3.50);
   \draw[draw=red!80!black, thick] (4.50,-3.50) -- (6.00,-2.00);
   \draw[draw=red!80!black, thick] (6.00,-2.00) -- (4.00,-1.00);
   \draw[draw=red!80!black, thick] (4.00,-1.00) -- (3.00,-2.00);
   \draw[draw=black, thin, solid] (-3.00,-2.00) -- (0.00,-1.00);
   \draw[draw=black, thin, solid] (-4.00,-1.00) -- (-1.00,1.00);
   \draw[draw=black, thin, solid] (4.00,-1.00) -- (1.00,1.00);
   \draw[draw=black, thin, solid] (3.00,-2.00) -- (0.00,-1.00);
   \draw[draw=black, thin, solid] (-4.50,-3.50) -- (4.50,-3.50);
   \draw[draw=black, thin, solid] (6.00,-2.00) -- (0.00,3.00);
   \draw[draw=red!80!black, ultra thick] (-6.00,-2.00) -- (-4.00,-1.00);
   \draw[draw=red!80!black, ultra thick] (-4.00,-1.00) -- (-3.00,-2.00);
   \draw[draw=red!80!black, ultra thick] (-3.00,-2.00) -- (-4.50,-3.50);
   \draw[draw=red!80!black, ultra thick] (-4.50,-3.50) -- (-6.00,-2.00);
   \draw[draw=red!80!black, ultra thick] (-1.00,1.00) -- (0.00,2.00);
   \draw[draw=red!80!black, ultra thick] (0.00,2.00) -- (1.00,1.00);
   \draw[draw=red!80!black, ultra thick] (1.00,1.00) -- (0.00,0.00);
   \draw[draw=red!80!black, ultra thick] (0.00,0.00) -- (-1.00,1.00);
   \draw[draw=red!80!black, ultra thick] (4.00,-1.00) -- (3.00,-2.00);
   \draw[draw=red!80!black, ultra thick] (3.00,-2.00) -- (4.50,-3.50);
   \draw[draw=red!80!black, ultra thick] (4.50,-3.50) -- (6.00,-2.00);
   \draw[draw=red!80!black, ultra thick] (6.00,-2.00) -- (4.00,-1.00);
   \draw[draw=black, very thick, solid] (-6.00,-2.00) -- (0.00,3.00);
   \draw[draw=black, very thick, solid] (0.00,3.00) -- (0.00,2.00);
   \draw[draw=black, very thick, solid] (-1.00,1.00) -- (-4.00,-1.00);
   \draw[draw=black, very thick, solid] (-3.00,-2.00) -- (0.00,-1.00);
   \draw[draw=black, very thick, solid] (0.00,-1.00) -- (0.00,0.00);
   \draw[draw=black, very thick, solid] (0.00,-1.00) -- (3.00,-2.00);
   \draw[draw=black, very thick, solid] (4.00,-1.00) -- (1.00,1.00);
   \draw[draw=black, very thick, solid] (6.00,-2.00) -- (0.00,3.00);
   \draw[draw=black, very thick, solid] (-4.50,-3.50) -- (4.50,-3.50);
   \draw[draw=black, very thick, solid] (0.00,3.00) circle (0.1);
   \draw[draw=black, fill=black, very thick, solid] (0.00,3.00) circle (0.1);
   \draw[draw=black, fill=black, very thick, solid] (0.00,2.00) circle (0.1);
   \draw[draw=black, fill=black, very thick, solid] (-1.00,1.00) circle (0.1);
   \draw[draw=black, fill=black, very thick, solid] (1.00,1.00) circle (0.1);
   \draw[draw=black, fill=black, very thick, solid] (0.00,0.00) circle (0.1);
   \draw[draw=black, fill=black, very thick, solid] (-3.00,-2.00) circle (0.1);
   \draw[draw=black, fill=black, very thick, solid] (-4.00,-1.00) circle (0.1);
   \draw[draw=black, fill=black, very thick, solid] (-6.00,-2.00) circle (0.1);
   \draw[draw=black, fill=black, very thick, solid] (-4.50,-3.50) circle (0.1);
   \draw[draw=black, fill=black, very thick, solid] (0.00,-1.00) circle (0.1);
   \draw[draw=black, fill=black, very thick, solid] (4.50,-3.50) circle (0.1);
   \draw[draw=black, fill=black, very thick, solid] (3.00,-2.00) circle (0.1);
   \draw[draw=black, fill=black, very thick, solid] (4.00,-1.00) circle (0.1);
   \draw[draw=black, fill=black, very thick, solid] (6.00,-2.00) circle (0.1);
\end{tikzpicture}
\caption{The Schlegel diagram of the auxiliary comparison polyhedron
\(\widehat B_\alpha\), obtained from \(P(3,2)\) by opening its three ideal
vertices into quadrilateral faces.  Every red edge has been
assigned angle \(\pi/2\), whether or not the corresponding quadrilateral was a genuine
cutting face of the original atoroidal component; all black edges have angle
\(\alpha\).}
\label{fig:p32-opened}
\end{figure}
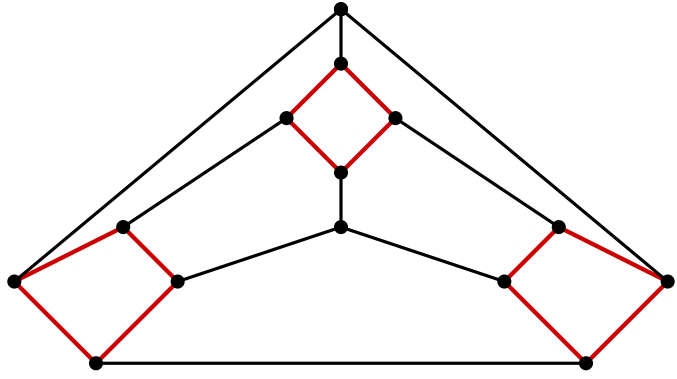

Geometrically, \(\widehat B_\alpha\) is the polar truncation of a generalized polyhedron obtained from the triangular bipyramid with three hyperideal vertices. The three quadrilateral faces are polar truncation faces; hence their edges are orthogonal to the adjacent old faces, in agreement with the angle \(\pi/2\) on these edges.

\begin{lemma}[The exceptional limit \(P(3,2)\)]\label{lem:Bhat-volume}
For all
\[
        \frac{\pi}{3}<\alpha<\arccos\frac13
\]
one has
\[
        \Vol(\widehat B_\alpha)>\vtet.
\]
Consequently, every atoroidal component with right-angled limit \(P(3,2)\) has volume greater than \(\Vol(T_\alpha)\).
\end{lemma}

\begin{proof}
The polyhedron \(\widehat B_\alpha\) has a symmetry group of order \(12\).
We describe one fundamental chamber explicitly.  Choose an old face plane
\(F\), inherited from a triangular face of \(P(3,2)\), one of its apical
edges \(h\), and its equatorial edge \(g\).
Let \(H_0\) be the symmetry plane interchanging the two apices, let \(H_1\)
be the symmetry plane bisecting \(F\), and let \(H_2\) be the symmetry plane
through the axis of the bipyramid and the edge \(h\).  The four planes
\[
        H_1,\qquad H_2,\qquad F,\qquad H_0
\]
bound a complete orthoscheme.  The planes \(H_1,H_2\) are consecutive
vertical mirrors of the triangular symmetry and meet at angle \(\pi/3\).
The mirror \(H_2\) interchanges the two old faces meeting at \(h\), whereas
\(H_0\) interchanges the two old faces meeting at \(g\); consequently
\[
        \angle(H_2,F)=\angle(F,H_0)=\frac{\alpha}{2}.
\]
Every other pair among the four bounding planes which is not consecutive in
the displayed order is orthogonal.  Hence the chamber has essential angles
\begin{equation}\label{eq:Bhat-orthoscheme-angles}
        \beta_{01}=\frac{\pi}{3},\qquad
        \beta_{12}=\frac{\alpha}{2},\qquad
        \beta_{23}=\frac{\alpha}{2}.
\end{equation}
The images of this chamber under the order-\(12\) symmetry group have
disjoint interiors and fill \(\widehat B_\alpha\).  Thus
\begin{equation}\label{eq:Bhat-volume-orthoscheme}
        \Vol(\widehat B_\alpha)=
        12\,\Vol\mathcal O\left(\frac{\pi}{3},\frac{\alpha}{2},\frac{\alpha}{2}\right),
\end{equation}
where the volume of the complete orthoscheme is computed by Kellerhals' formula \eqref{eq:kellerhals}.

\begin{figure}[ht]
\centering
\begin{tikzpicture}[scale=0.95, line join=round, line cap=round,
    v/.style={circle, fill=black, inner sep=1.6pt},
    old/.style={very thick, black},
    aux/.style={thick, gray, dashed},
    cut/.style={very thick, red},
    every node/.style={font=\small}]
	\draw[draw=black, very thick, solid] (-2.00,-2.00) -- (2.00,3.00);
	\draw[draw=black, very thick, solid] (2.00,3.00) -- (2.00,-1.00);
	\draw[draw=black, very thick, solid] (-2.00,-2.00) -- (-0.50,-2.50);
	\draw[draw=black, very thick, solid] (1.50,-2.50) -- (2.00,-1.00);
	\draw[draw=black, very thick, solid] (0.50,-0.50) -- (2.00,3.00);
	\draw[draw=red!80!black, ultra thick] (-0.50,-2.50) -- (0.50,-0.50);
	\draw[draw=red!80!black, ultra thick] (0.50,-0.50) -- (1.50,-2.50);
	\draw[draw=red!80!black, ultra thick] (1.50,-2.50) -- (-0.50,-2.50);
	\draw[draw=black, very thick, dashed] (-2.00,-2.00) -- (2.00,-1.00);
	\draw[draw=red!80!black, ultra thick] (0.50,-0.50) -- (-0.50,-2.50);
	\draw[draw=red!80!black, ultra thick] (-0.50,-2.50) -- (1.50,-2.50);
	\draw[draw=black, thin, solid] (2.00,3.00) circle (0.1);
\draw[draw=black, fill=black, very thick, solid] (2.00,3.00) circle (0.1);
\draw[draw=black, fill=black, very thick, solid] (2.00,-1.00) circle (0.1);
\draw[draw=black, fill=black, very thick, solid] (-2.00,-2.00) circle (0.1);
\draw[draw=black, fill=black, very thick, solid] (-0.50,-2.50) circle (0.1);
\draw[draw=black, fill=black, very thick, solid] (1.50,-2.50) circle (0.1);
\draw[draw=black, fill=black, very thick, solid] (0.50,-0.50) circle (0.1);

	\node[black, anchor=south west] at (1.94,0.75) {$\frac{\pi}{3}$};
	\node[black, anchor=south west] at (-1.78,-2.95) {$\frac{\alpha}{2}$};
	\node[black, anchor=south west] at (0.94,0.25) {$\frac{\alpha}{2}$};
\end{tikzpicture}
\caption{A fundamental complete orthoscheme in the decomposition of
$\widehat B_\alpha$. Its essential angles are
$\pi/3,\alpha/2,\alpha/2$. The hyperideal vertex is cut off by its polar
plane; the edges of the new truncation face are red.}
\label{fig:bhat-orthoscheme}
\end{figure}
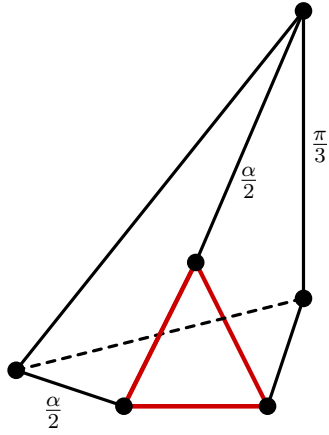

By the Schl\"afli formula, the volume of \(\widehat B_\alpha\) decreases as
\(\alpha\) increases. Therefore its infimum on the open interval under
consideration is the limiting value at
\[
        \alpha_0=\arccos\frac13\approx1.230959417 .
\]
We certify the numerical comparison rather than relying on a decimal
approximation alone.  Recall the absolutely convergent Fourier expansion
\[
 \Lambda(x)=\frac12\sum_{k=1}^{\infty}\frac{\sin(2kx)}{k^2}.
\]
If \(\Lambda_N\) denotes the sum through \(k=N\), then uniformly in \(x\)
\begin{equation}\label{eq:lobachevsky-tail}
 \left|\Lambda(x)-\Lambda_N(x)\right|
 \le \frac12\sum_{k>N}\frac1{k^2}<\frac1{2N}.
\end{equation}
For \(\alpha_0=\arccos(1/3)\), the parameter \(\theta_0\) in
\eqref{eq:kellerhals} satisfies
\[
 1.2309594173<\alpha_0<1.2309594174,
 \qquad
 1.0068536854<\theta_0<1.0068536855.
\]
Substitution of these rational intervals into the finite sum with \(N=100\),
using directed rounding for the sine terms, gives
\[
 3.3047<
 12\,\Vol_{100}\mathcal O
 \left(\frac{\pi}{3},\frac{\alpha_0}{2},
                    \frac{\alpha_0}{2}\right)
 <3.3048.
\]
Here \(\Vol_{100}\) means that every occurrence of \(\Lambda\) in
\eqref{eq:kellerhals} is replaced by \(\Lambda_{100}\).
The sum of the absolute coefficients of the Lobachevsky terms in
\eqref{eq:kellerhals} is \(8\).  After multiplying by \(12/4\),
\eqref{eq:lobachevsky-tail} therefore bounds the total truncation error by
\(12/100=0.12\).  Consequently
\[
        \Vol(\widehat B_{\alpha_0})>3.1847.
\]
On the other hand,
\[
 \vtet=3\Lambda\left(\frac{\pi}{3}\right)
 \le \frac{\pi^2}{4}<2.47,
\]
where the first inequality follows at once by taking absolute values in the
same Fourier series.  Thus
\(\Vol(\widehat B_{\alpha_0})>\vtet\).  For reference, a higher-precision
evaluation is
\[
        \Vol(\widehat B_{\alpha_0})=3.305902996\ldots .
\]
It follows that
\[
        \Vol(\widehat B_\alpha)>\vtet
\]
for all \(\alpha\in(\pi/3,\alpha_0)\).  Together with
\eqref{eq:component-to-Bhat}, every atoroidal component with
right-angled limit \(P(3,2)\) has volume greater than
\(\vtet>\Vol(T_\alpha)\).
\end{proof}

\section{The graph-type case}\label{sec:prismatic}

Throughout the applications in this section, \(P_\alpha\) is an equiangular
compact hyperbolic polyhedron with
\[
        \frac{\pi}{3}<\alpha<\alpha_0=\arccos\frac13,
\]
and Atkinson's decomposition is assumed to have no atoroidal component. Hence
\(P_\alpha\) has graph type. We shall reduce it, through volume-nonincreasing
deformations, to a single prism. The geometric move is a direct equiangular
version of Inoue's edge surgery: only the angle of the seam being removed is
increased, while every surviving angle remains equal to \(\alpha\).  The
surgery lemma itself is proved in the full range
\(\pi/3<\alpha<\pi/2\).

\subsection{Prismatic trees and leaf blocks}

A prism block \(D\) of degree \(n\) is written combinatorially as
\[
        D=Q_n\times I.
\]
Its two faces \(Q_n\times\{0\}\) and \(Q_n\times\{1\}\) are its \emph{bases}.
The remaining \(n\) quadrilateral faces are its \emph{side faces}. On a side
face there are two opposite base-direction edges, contained in the bases, and
two opposite lateral-direction edges, each shared by two consecutive side
faces.

Suppose two prism blocks are glued along quadrilateral cutting faces. Relative
to their product structures, the gluing either preserves or interchanges the
two pairs of opposite edges of the attaching quadrilateral. In the first case
the union is again a single prism, so this gluing does not give two distinct
vertices of the prismatic tree. We shall therefore use the tree only for the
second, \emph{crossed}, gluings. Geometrically, a crossed gluing interchanges
the base and lateral directions on the attaching quadrilateral; this is what
is meant below by turning one prism through a quarter turn.

A \(4\)-prism attached as a leaf is combinatorially absorbed by the adjacent
prism: the attachment does not change the abstract polyhedron. Thus, once the
surgeries below have reduced a leaf block to a \(4\)-prism, that block no
longer occurs as a separate vertex of the prismatic tree.
Consequently every leaf retained in the reduced tree has degree at least five.
This absorption is shown in Figure~\ref{fig:four-prism-absorption}.

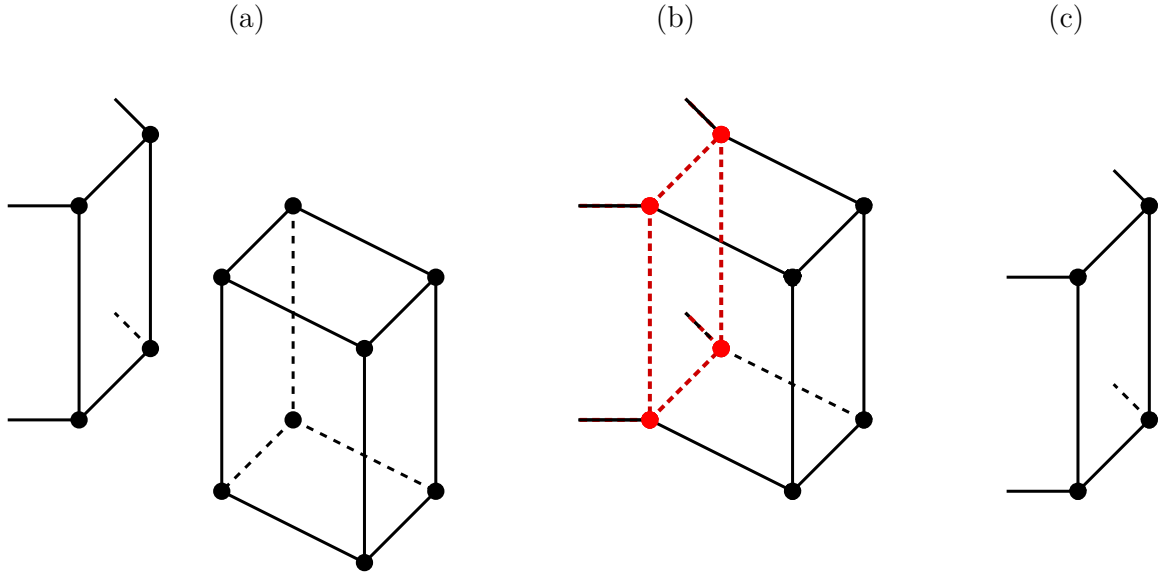
\begin{figure}[ht]
\centering
\resizebox{0.96\linewidth}{!}{%
\begin{tikzpicture}
	\draw[draw=black, fill=black, very thick, solid] (-6.00,2.00) circle (0.1);
	\draw[draw=black, fill=black, very thick, solid] (-5.00,3.00) circle (0.1);
	\draw[draw=black, fill=black, very thick, solid] (-6.00,-1.00) circle (0.1);
	\draw[draw=black, fill=black, very thick, solid] (-5.00,0.00) circle (0.1);
	\draw[draw=black, very thick, solid] (-6.00,2.00) -- (-6.00,-1.00);
	\draw[draw=black, very thick, solid] (-6.00,2.00) -- (-5.00,3.00);
	\draw[draw=black, very thick, solid] (-5.00,3.00) -- (-5.00,0.00);
	\draw[draw=black, very thick, solid] (-6.00,-1.00) -- (-5.00,0.00);
	\draw[draw=black, very thick, solid] (-6.00,-1.00) -- (-7.00,-1.00);
	\draw[draw=black, very thick, solid] (-6.00,2.00) -- (-7.00,2.00);
	\draw[draw=black, very thick, solid] (-5.00,3.00) -- (-5.50,3.50);
	\draw[draw=black, fill=black, very thick, solid] (-4.00,-2.00) circle (0.1);
	\draw[draw=black, fill=black, very thick, solid] (-3.00,-1.00) circle (0.1);
	\draw[draw=black, fill=black, very thick, solid] (-4.00,1.00) circle (0.1);
	\draw[draw=black, fill=black, very thick, solid] (-3.00,2.00) circle (0.1);
	\draw[draw=black, fill=black, very thick, solid] (-2.00,-3.00) circle (0.1);
	\draw[draw=black, fill=black, very thick, solid] (-1.00,-2.00) circle (0.1);
	\draw[draw=black, fill=black, very thick, solid] (-2.00,0.00) circle (0.1);
	\draw[draw=black, fill=black, very thick, solid] (-1.00,1.00) circle (0.1);
	\draw[draw=black, very thick, solid] (-4.00,1.00) -- (-4.00,-2.00);
	\draw[draw=black, very thick, solid] (-4.00,1.00) -- (-3.00,2.00);
	\draw[draw=black, very thick, solid] (-4.00,1.00) -- (-2.00,0.00);
	\draw[draw=black, very thick, solid] (-2.00,0.00) -- (-2.00,-3.00);
	\draw[draw=black, very thick, solid] (-2.00,-3.00) -- (-1.00,-2.00);
	\draw[draw=black, very thick, solid] (-2.00,-3.00) -- (-4.00,-2.00);
	\draw[draw=black, very thick, solid] (-1.00,1.00) -- (-1.00,-2.00);
	\draw[draw=black, very thick, solid] (-3.00,2.00) -- (-1.00,1.00);
	\draw[draw=black, very thick, solid] (-2.00,0.00) -- (-1.00,1.00);
	\draw[draw=black, very thick, dashed] (-3.00,-1.00) -- (-1.00,-2.00);
	\draw[draw=black, very thick, dashed] (-3.00,-1.00) -- (-4.00,-2.00);
	\draw[draw=black, very thick, dashed] (-3.00,2.00) -- (-3.00,-1.00);
	\draw[draw=black, very thick, dashed] (-5.00,0.00) -- (-5.50,0.50);
	\draw[draw=black, fill=black, very thick, dashed] (2.00,-1.00) circle (0.1);
	\draw[draw=black, fill=black, very thick, dashed] (3.00,0.00) circle (0.1);
	\draw[draw=black, fill=black, very thick, dashed] (2.00,2.00) circle (0.1);
	\draw[draw=black, fill=black, very thick, dashed] (3.00,3.00) circle (0.1);
	\draw[draw=black, fill=black, very thick, dashed] (4.00,-2.00) circle (0.1);
	\draw[draw=black, fill=black, very thick, dashed] (5.00,-1.00) circle (0.1);
	\draw[draw=black, fill=black, very thick, dashed] (4.00,1.00) circle (0.1);
	\draw[draw=black, fill=black, very thick, dashed] (5.00,2.00) circle (0.1);
	\draw[draw=black, very thick, solid] (2.00,2.00) -- (4.00,1.00);
	\draw[draw=black, very thick, solid] (4.00,1.00) -- (5.00,2.00);
	\draw[draw=black, very thick, solid] (5.00,2.00) -- (3.00,3.00);
	\draw[draw=black, very thick, solid] (4.00,-2.00) -- (2.00,-1.00);
	\draw[draw=black, very thick, solid] (4.00,-2.00) -- (5.00,-1.00);
	\draw[draw=black, very thick, solid] (4.00,-2.00) -- (4.00,1.00);
	\draw[draw=black, very thick, solid] (5.00,2.00) -- (5.00,-1.00);
	\draw[draw=black, very thick, dashed] (3.00,0.00) -- (5.00,-1.00);
	\draw[draw=black, very thick, solid] (2.00,2.00) -- (1.00,2.00);
	\draw[draw=black, very thick, solid] (2.00,-1.00) -- (1.00,-1.00);
	\draw[draw=black, very thick, solid] (3.00,3.00) -- (2.50,3.50);
	\draw[draw=black, very thick, dashed] (3.00,0.00) -- (2.50,0.50);
	\draw[draw=black, fill=black, very thick, dashed] (8.00,-2.00) circle (0.1);
	\draw[draw=black, fill=black, very thick, dashed] (9.00,-1.00) circle (0.1);
	\draw[draw=black, fill=black, very thick, dashed] (8.00,1.00) circle (0.1);
	\draw[draw=black, fill=black, very thick, dashed] (9.00,2.00) circle (0.1);
	\draw[draw=black, very thick, dashed] (9.00,-1.00) -- (8.50,-0.50);
	\draw[draw=black, very thick, solid] (8.00,1.00) -- (8.00,-2.00);
	\draw[draw=black, very thick, solid] (8.00,-2.00) -- (9.00,-1.00);
	\draw[draw=black, very thick, solid] (9.00,-1.00) -- (9.00,2.00);
	\draw[draw=black, very thick, solid] (8.00,-2.00) -- (7.00,-2.00);
	\draw[draw=black, very thick, solid] (8.00,1.00) -- (7.00,1.00);
	\draw[draw=black, very thick, solid] (9.00,2.00) -- (8.50,2.50);
	\draw[draw=black, very thick, solid] (8.00,1.00) -- (9.00,2.00);
	\node[black, anchor=south west] at (-4.06,4.25) {(a)};
	\node[black, anchor=south west] at (1.94,4.25) {(b)};
	\node[black, anchor=south west] at (7.44,4.25) {(c)};
	\draw[draw=red, fill=red, very thick, solid] (2.00,-1.00) circle (0.1);
	\draw[draw=red, fill=red, very thick, solid] (3.00,0.00) circle (0.1);
	\draw[draw=red, fill=red, very thick, solid] (2.00,2.00) circle (0.1);
	\draw[draw=red, fill=red, very thick, solid] (3.00,3.00) circle (0.1);
	\draw[draw=red!80!black, ultra thick, densely dashed] (2.00,2.00) -- (2.00,-1.00);
	\draw[draw=red!80!black, ultra thick, densely dashed] (2.00,2.00) -- (3.00,3.00);
	\draw[draw=red!80!black, ultra thick, densely dashed] (2.00,-1.00) -- (1.00,-1.00);
	\draw[draw=red!80!black, ultra thick, densely dashed] (2.00,2.00) -- (1.00,2.00);
	\draw[draw=red!80!black, ultra thick, densely dashed] (3.00,3.00) -- (2.50,3.50);
	\draw[draw=red!80!black, ultra thick, densely dashed] (3.00,0.00) -- (2.50,0.50);
	\draw[draw=red!80!black, ultra thick, densely dashed] (2.00,-1.00) -- (3.00,0.00);
	\draw[draw=red!80!black, ultra thick, densely dashed] (3.00,3.00) -- (3.00,0.00);
	\draw[draw=black, very thick, solid] (2.00,2.00) -- (1.00,2.00);
	\draw[draw=black, very thick, solid] (3.00,3.00) -- (2.50,3.50);
	\draw[draw=black, very thick, solid] (2.00,-1.00) -- (1.00,-1.00);
	\draw[draw=black, very thick, dashed] (3.00,0.00) -- (2.50,0.50);
	\draw[draw=red, fill=red, very thick, dashed] (3.00,0.00) circle (0.1);
	\draw[draw=red, fill=red, very thick, dashed] (2.00,-1.00) circle (0.1);
	\draw[draw=red, fill=red, very thick, dashed] (2.00,2.00) circle (0.1);
	\draw[draw=red, fill=red, very thick, dashed] (3.00,3.00) circle (0.1);
	\draw[draw=red, fill=red, very thick, dashed] (2.00,2.00) circle (0.1);
	\draw[draw=black, very thick, dashed] (4.00,1.00) circle (0.1);
	\draw[draw=black, very thick, dashed] (4.00,1.00) circle (0.1);
	\draw[draw=black, ultra thick, dashed] (4.00,1.00) circle (0.1);
	\draw[draw=black, fill=black, very thick, dashed] (4.00,1.00) circle (0.1);
	\draw[draw=black, fill=black, ultra thick, dashed] (4.00,1.00) circle (0.1);
	\draw[draw=black, fill=black, ultra thick, dashed] (4.00,1.00) circle (0.1);
	\draw[draw=black, fill=black, ultra thick, dashed] (4.00,1.00) circle (0.1);
	\draw[draw=black, fill=black, very thick, solid] (4.00,1.00) circle (0.1);
	\draw[draw=black, fill=black, very thick, solid] (5.00,2.00) circle (0.1);
	\draw[draw=black, fill=black, very thick, solid] (5.00,-1.00) circle (0.1);
	\draw[draw=black, fill=black, very thick, solid] (4.00,-2.00) circle (0.1);
	\draw[draw=black, fill=black, very thick, solid] (8.00,-2.00) circle (0.1);
	\draw[draw=black, fill=black, very thick, solid] (8.00,1.00) circle (0.1);
	\draw[draw=black, fill=black, very thick, solid] (9.00,2.00) circle (0.1);
	\draw[draw=black, fill=black, very thick, solid] (9.00,-1.00) circle (0.1);
	\draw[draw=red, fill=red, very thick, solid] (2.00,2.00) circle (0.1);
	\draw[draw=red, fill=red, very thick, solid] (3.00,3.00) circle (0.1);
	\draw[draw=red, fill=red, very thick, solid] (2.00,-1.00) circle (0.1);
	\draw[draw=red, fill=red, very thick, solid] (3.00,0.00) circle (0.1);
\end{tikzpicture}
}
\caption{Combinatorial absorption of a leaf \(4\)-prism.  In
\textup{(a)} the leaf \(4\)-prism and the adjacent prism are shown before
gluing.  In \textup{(b)} they have been identified along the crossed
quadrilateral; the part that is suppressed is thick and dash-patterned. After removing the
internal gluing face, one
obtains \textup{(c)}, which is combinatorially the original adjacent prism.}
\label{fig:four-prism-absorption}
\end{figure}

For later use we record the free faces of a leaf directly. A leaf has one
attaching side face.  If its gluing is crossed, that attaching face disappears
and the two neighboring side faces merge with the two bases of the adjacent
block.  Removing these three consecutive side positions from the cyclic list
of the \(n\) side faces leaves exactly
\begin{equation}\label{eq:leaf-free-chain}
        n-3
\end{equation}
free quadrilateral faces, sharing edges successively in one linear chain.  This
is the chain on which the surgeries below are performed.  For a leaf of the
initial reduced decomposition, Atkinson and Rafalski prove in addition that
\(n\ge5\) \cite[Lemma~3.5]{AtkinsonRafalski}.  The elementary count of the
free chain, unlike the canonicality statement, remains valid for the coarser
prism descriptions obtained later by absorbing \(4\)-prisms.

\begin{lemma}[No prismatic \(3\)-circuits in a prism tree]
\label{lem:prism-tree-no-three}
Let \(P\) be a trivalent abstract polyhedron obtained by gluing a finite tree
of prisms \(Q_{n_v}\times I\), where \(n_v\ge4\), along quadrilateral side
faces and suppressing every gluing face.  Then \(P\) has no prismatic
\(3\)-circuit.
\end{lemma}

\begin{proof}
Pass to the dual triangulation \(\Delta=P^*\).  Since \(P\) is trivalent, a
prismatic \(3\)-circuit in \(P\) is the same thing as a \(3\)-cycle in
\(\Delta\) which is not the boundary of a triangular face.  We prove that no
such cycle occurs, by induction on the number of prism blocks.

For one block, \(\Delta\) is the \(n\)-bipyramid, where \(n\ge4\).  Its only
\(3\)-cycles are its triangular faces.  (For \(n=3\), the equatorial
triangle would be a nonfacial \(3\)-cycle; this is exactly why triangular
prisms are excluded.)

Now remove a leaf block from the tree.  Let \(P'\) be the polyhedron formed
by the remaining blocks, and let \(D=Q_n\times I\) be the leaf prism.  In
the dual triangulations \((P')^*\) and \(D^*\), the two quadrilateral faces
which are to be glued correspond to vertices \(x\) and \(y\) of valence
four.  The dual of the glued polyhedron is obtained by deleting the open
stars of \(x\) and \(y\) and identifying their boundary \(4\)-cycles.

Both boundary \(4\)-cycles are chordless.  This is immediate on the
bipyramid side.  On the \((P')^*\) side, a chord together with \(x\) would
form a nonfacial \(3\)-cycle, contrary to the induction hypothesis.

Consider a \(3\)-cycle after the two boundary cycles have been identified.
If it contains a vertex away from the common \(4\)-cycle, then all its
vertices lie on one side of the gluing.  It was therefore already a
\(3\)-cycle on that side and, by induction, bounds a face.  Otherwise all
three vertices lie on the common \(4\)-cycle, but three vertices of a
chordless \(4\)-cycle cannot form a triangle.  Thus every \(3\)-cycle in
\(\Delta\) is facial, and \(P\) has no prismatic \(3\)-circuit.
\end{proof}

\subsection{Deleting a lateral seam}

Let \(P\) be a trivalent abstract polyhedron, and let \(e\) be an edge. Let
\(A\) and \(B\) be the two faces containing \(e\), and let \(u,v\) be its
endpoints. At \(u\) there is a third face \(C\), and at \(v\) there is a third
face \(D\). Deleting \(e\) means removing \(e\), merging \(A\) and \(B\), and
demoting \(u,v\). Thus the two edges of \(C\) ending at \(u\) concatenate to
one edge, and the two edges of \(D\) ending at \(v\) concatenate to one edge.
The corresponding change in the degrees of the four incident faces is shown
in Figure~\ref{fig:edge-deletion-face-degrees}.

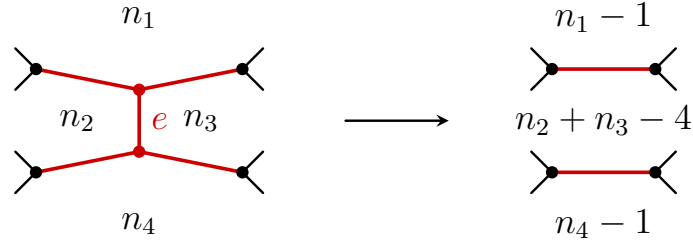
\begin{figure}[ht]
\centering
\resizebox{0.57\linewidth}{!}{%
\begin{tikzpicture}[scale=1.2,line cap=round,line join=round]
  \coordinate (p1) at (0,0);
  \coordinate (p2) at (2,0);
  \coordinate (p3) at (0,1);
  \coordinate (p4) at (2,1);
  \draw[line width=1.2pt,red!80!black] (0,1)--(1,0.8)--(2,1);
  \draw[line width=1.2pt,red!80!black] (0,0)--(1,0.2)--(2,0);
  \draw[line width=1.2pt,red!80!black] (1,0.2)--(1,0.8);
  \foreach \p in {(p1),(p2),(p3),(p4)} \fill[black] \p circle (1.7pt);
  \fill[red!80!black] (1,0.2) circle (1.7pt);
  \fill[red!80!black] (1,0.8) circle (1.7pt);
  \draw[thick] (0,1)--(-0.2,1.2) (0,1)--(-0.2,0.8)
               (2,1)--(2.2,1.2) (2,1)--(2.2,0.8)
               (2,0)--(2.2,0.2) (2,0)--(2.2,-0.2)
               (0,0)--(-0.2,-0.2) (0,0)--(-0.2,0.2);
  \node at (1,1.5) {$n_1$};
  \node at (0.4,0.5) {$n_2$};
  \node at (1.6,0.5) {$n_3$};
  \node at (1,-0.5) {$n_4$};
  \node[red!80!black] at (1.2,0.5) {$e$};

  \draw[-stealth,thick] (3,0.5)--(4,0.5);

  \coordinate (q1) at (5,0);
  \coordinate (q2) at (6,0);
  \coordinate (q3) at (5,1);
  \coordinate (q4) at (6,1);
  \draw[line width=1.2pt,red!80!black] (q1)--(q2) (q3)--(q4);
  \foreach \p in {(q1),(q2),(q3),(q4)} \fill[black] \p circle (1.7pt);
  \draw[thick] (5,1)--(4.8,1.2) (5,1)--(4.8,0.8)
               (6,1)--(6.2,1.2) (6,1)--(6.2,0.8)
               (6,0)--(6.2,0.2) (6,0)--(6.2,-0.2)
               (5,0)--(4.8,-0.2) (5,0)--(4.8,0.2);
  \node at (5.5,1.5) {$n_1-1$};
  \node at (5.5,0.5) {$n_2+n_3-4$};
  \node at (5.5,-0.5) {$n_4-1$};
\end{tikzpicture}%
}
\caption{Local combinatorics of edge deletion. Before deletion the four
faces around \(e\) have \(n_1,n_2,n_3,n_4\) sides. Removing \(e\) merges the
two incident faces and suppresses its endpoints. Consequently the merged
face has \(n_2+n_3-4\) sides, while the other two faces have
\(n_1-1\) and \(n_4-1\) sides.}
\label{fig:edge-deletion-face-degrees}
\end{figure}

In our application, \(e\) is a lateral seam between two consecutive side
quadrilaterals. Their union after deletion is again a quadrilateral. The two
base-direction edges at the upper endpoint concatenate into the upper edge of
the new quadrilateral, and the analogous two lower edges concatenate into its
lower edge.

\begin{lemma}[Prismatic circuits under edge deletion]\label{lem:circuit-correspondence}
Assume that \(P\) has no prismatic \(3\)-circuit and that deleting \(e\)
produces an abstract polyhedron \(P\setminus e\). Then the following are
equivalent:
\begin{enumerate}[label=\textup{(\alph*)}]
\item \(e\) belongs to a prismatic \(4\)-circuit of \(P\);
\item \(P\setminus e\) has a prismatic \(3\)-circuit which was not present in
\(P\).
\end{enumerate}
In particular, if \(P\setminus e\) has no prismatic \(3\)-circuit, then \(e\)
belongs to no prismatic \(4\)-circuit of \(P\).
\end{lemma}

\begin{proof}
Pass to the dual triangulation \(\Delta=P^*\).  If \(A,B\) are the faces
incident to \(e\), then deleting \(e\) is dual to contracting the edge
\(AB\).  The two triangles \(ABC\) and \(ABD\) incident to \(AB\) disappear,
and the two pairs \(AC,BC\) and \(AD,BD\) become the single edges \(EC\) and
\(ED\), where \(E\) is the contracted vertex.

For a trivalent polyhedron, a prismatic \(3\)-circuit is exactly a
nonfacial \(3\)-cycle in the dual triangulation.  Likewise, a prismatic
\(4\)-circuit containing \(e\) is a chordless \(4\)-cycle in \(\Delta\)
containing \(AB\).

Let \(AUVBA\) be such a \(4\)-cycle, with \(AB\) as its fourth edge.
Contracting \(AB\) produces the \(3\)-cycle \(EUV E\).  It cannot bound a
triangle of the contracted triangulation, since its lift would add a
diagonal to the chordless \(4\)-cycle.  Hence it is a new prismatic
\(3\)-circuit.

Conversely, let \(EUV E\) be a new nonfacial \(3\)-cycle after the
contraction.  The two edges incident to \(E\) cannot both lift to edges at
\(A\), since then \(AUV A\) would be a \(3\)-cycle in \(\Delta\); it would
either be a pre-existing prismatic \(3\)-cycle or a face and hence would make
\(EUV E\) facial.  Both alternatives are excluded.  The same applies to
\(B\).  Thus one edge lifts at \(A\) and the other at \(B\), and together
with \(AB\) and \(UV\) they form a \(4\)-cycle through \(AB\).  A diagonal of
this cycle would again give a pre-existing \(3\)-cycle or make \(EUV E\)
facial.  Therefore the lifted \(4\)-cycle is chordless, hence prismatic.
\end{proof}

\subsection{Direct equiangular edge surgery}\label{subsec:direct-edge-surgery}

We now prove the deformation which will be used at every step of the
graph-type reduction.  As in Inoue's right-angled surgery, only the angle at
the active edge varies.  The new difficulty is that the four fixed polar
sides adjacent to that edge have length \(L=\pi-\alpha\), rather than
\(\pi/2\), so the affected two-triangle patch is not a spherical bigon.

\begin{lemma}[Direct equiangular edge surgery]\label{lem:direct-equiangular-edge-surgery}
Let \(\pi/3<\alpha<\pi/2\), let \(P_0\) be a compact equiangular hyperbolic
polyhedron with common angle \(\alpha\), and let \(e\) be an edge.  Let
\(A,B\) be the faces containing
\(e\), and let \(C,D\) be the third faces at its endpoints.  Assume:
\begin{enumerate}[label=\textup{(S\arabic*)},itemsep=0.15em,topsep=0.3em]
\item \(A\) and \(B\) are quadrilaterals;
\item deleting \(e\) gives an abstract polyhedron \(P_1\) which is realized as
a compact equiangular hyperbolic polyhedron with angle \(\alpha\);
\item \(e\) belongs to no prismatic \(4\)-circuit of \(P_0\);
\end{enumerate}
Then there is a family \(P_t\), \(\alpha\le t<\pi\), in which only the angle at
\(e\) changes and
\[
        \theta_e(P_t)=t.
\]
As \(t\to\pi\), the family converges to \(P_1\), and
\[
        \Vol(P_0)>\Vol(P_1).
\]
\end{lemma}

\begin{proof}
We give the polar construction in detail.  From now on the same capital letter
denotes a face of the primal polyhedron and the corresponding vertex of the
polar metric.  Put
\[
        s=\pi-t,\qquad L=\pi-\alpha .
\]
The two polar triangles at the endpoints of \(e\) are
\[
        T_C=CAB,\qquad T_D=DAB.
\]
They have the common side \(AB\) of length \(s\), while
\[
        |CA|=|CB|=|DA|=|DB|=L.
\]
Let
\[
        K_s=T_C\cup_{AB}T_D .
\]
Thus \(K_s\) is the union of two isosceles spherical triangles, not a bigon
unless \(L=\pi/2\).  This is precisely where the present argument differs from
Inoue's right-angled proof.

After \(e\) is deleted, the faces \(A\) and \(B\) merge to a face \(E\).
The polar metric \(P_1^\circ\) is cut along the two-edge path \(CED\), with
both edges of length \(L\).  Inserting \(K_s\) between the two copies of that
path separates \(E\) into \(A\) and \(B\).  Denote the resulting cone metric
by \(M_s\).  Every old polar edge has length \(L\), and the single new edge
\(AB\) has length \(s\).  Consequently its prescribed primal angles are
\(\alpha\) at every edge other than \(e\), and \(\pi-s=t\) at \(e\).

At this stage \(M_s\) is only a candidate polar metric.  In particular, no
statement which assumes that a metric is already the polar of a hyperbolic
polyhedron will be applied directly to \(M_s\).  Assumption \textup{(S2)}, on
the other hand, says that \(P_1\) is an actual compact equiangular polyhedron;
therefore its polar \(P_1^\circ\) is available throughout the argument.

As in Inoue's proof, there is a natural length-nonincreasing collapse map
from the metric before the surgery to the polar of the surgered polyhedron.
The formula for the map is slightly different because \(K_s\) is not a
right-angled bigon.  Parametrize the two edges of the path \(CED\subset
P_1^\circ\) by unit speed,
\[
 \iota_C\colon[0,L]\longrightarrow CE,
 \qquad
 \iota_D\colon[0,L]\longrightarrow DE,
\]
where
\[
 \iota_C(0)=C,\quad \iota_C(L)=E,
 \qquad
 \iota_D(0)=D,\quad \iota_D(L)=E.
\]
For \(p\in T_C\) and \(p\in T_D\), respectively, put
\begin{equation}\label{eq:collapse-map-on-patch}
 \begin{aligned}
 f(p)&=\iota_C\bigl(\min\{d_{T_C}(C,p),L\}\bigr),
       &&p\in T_C,\\
 f(p)&=\iota_D\bigl(\min\{d_{T_D}(D,p),L\}\bigr),
       &&p\in T_D.
 \end{aligned}
\end{equation}
These two definitions agree on \(AB\).  Use the coordinates from
Lemma~\ref{lem:reflex-patch-crossing}.  A point \(p\in AB\) has the form
\[
        p=(\cos u,\sin u,0),\qquad |u|\le \frac{s}{2}.
\]
Since \(r=\cos(s/2)\), \(c=\cos L<0\), and \(\cos u\ge r\), we have
\[
 \cos d(C,p)=\frac{c}{r}\cos u\le c=\cos L.
\]
Hence \(d(C,p)\ge L\).  The same calculation with \(D\) gives
\(d(D,p)\ge L\).  Thus both formulas in
\eqref{eq:collapse-map-on-patch} send all of \(AB\) to \(E\).  On the four
outer sides they give the natural identifications
\[
        CA,CB\longrightarrow CE,
        \qquad
        DA,DB\longrightarrow DE.
\]
Consequently they extend, by the natural cellwise isometry outside \(K_s\),
to a continuous map
\begin{equation}\label{eq:collapse-map}
        f\colon M_s\longrightarrow P_1^\circ .
\end{equation}

The map \(f\) is \(1\)-Lipschitz.  On \(T_C\), the function
\(p\mapsto d(C,p)\) is \(1\)-Lipschitz, truncation at \(L\) does not increase
distances, and \(\iota_C\) is parametrized by arclength.  The same argument
applies on \(T_D\), while on every unchanged triangle \(f\) is an isometry.
Subdividing any rectifiable path into pieces contained in these triangles
shows that its image has no greater length.  In particular,
\begin{equation}\label{eq:collapse-map-length}
        \ell(f\circ\eta)\le \ell(\eta)
\end{equation}
for every rectifiable path \(\eta\subset M_s\).  When \(L=\pi/2\), formula
\eqref{eq:collapse-map-on-patch} is exactly Inoue's projection of the
spherical bigon onto its leaf space.

We verify the three conditions of Theorem~\ref{thm:rhs}.

\emph{Condition \textup{(RH1)}.}
The metric \(M_s\) is obtained by gluing finitely many spherical triangles.
It therefore has curvature \(+1\) away from its finitely many vertices.

\emph{Condition \textup{(RH2)}.}
Let \(\delta=\delta(s)\), \(\gamma=\gamma(s)\), and let \(\beta\) be as in
\eqref{eq:polar-L-beta}.  Cone angles away from \(A,B,C,D\) do not change.
Because \(A\) and \(B\) are quadrilaterals, the star of either one contains the
two special triangles and two unchanged equilateral triangles.  Hence
\[
        \omega(A)=\omega(B)=2\delta+2\beta
        >\pi+\pi=2\pi .
\]
At \(C\) and \(D\), deleting \(K_s\) removes exactly the positive angle
\(\gamma\).  Therefore
\[
        \omega_{M_s}(C)=\omega_{P_1^\circ}(C)+\gamma>2\pi,
        \qquad
        \omega_{M_s}(D)=\omega_{P_1^\circ}(D)+\gamma>2\pi .
\]
All cone angles of \(M_s\) are thus strictly greater than \(2\pi\).

\emph{Condition \textup{(RH3)}.}
We first make explicit the vertices denoted by \(X\) and \(Y\).  Since \(A\)
is a quadrilateral, its four neighboring faces occur cyclically as
\[
        B,\ C,\ X,\ D.
\]
Equivalently, the two ordinary polar triangles adjacent across \(CA\) and
\(DA\) are \(CAX\) and \(DAX\).  The face \(X\) is the face across the edge of
\(A\) opposite \(e\).  It may be a quadrilateral, but it may also be the large
base face created by a crossed prism gluing.  Accordingly,
\(\operatorname{st}(X)\) may contain any number \(m\ge4\) of triangles.  The
vertex \(Y\) is defined in exactly the same way from the cyclic list
\[
        A,\ C,\ Y,\ D
\]
of faces adjacent to \(B\).  Assumption \textup{(S3)} implies that \(X\ne Y\)
and that \(X\) and \(Y\) are not adjacent.  Indeed, equality would give a
prismatic \(3\)-circuit, which is impossible for an equiangular polyhedron
with \(3\alpha>\pi\), while an edge \(XY\) would close a prismatic
\(4\)-circuit containing \(e\), contrary to \textup{(S3)}.  No assumption on
the valences of \(X\) and \(Y\) will be used.

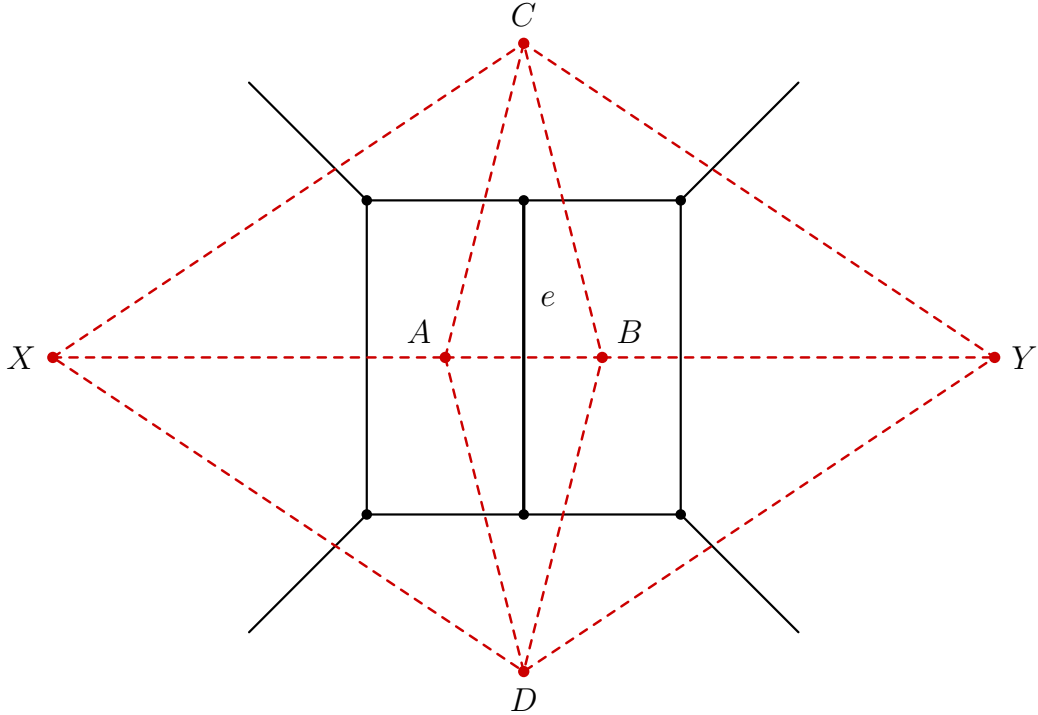
\begin{figure}[H]
\centering
\resizebox{0.87\linewidth}{!}{%
\begin{tikzpicture}[line cap=round,line join=round,
  primal/.style={black,line width=0.8pt},
  polar/.style={red!80!black,line width=0.9pt,dashed},
  pvertex/.style={circle,fill=black,inner sep=1.35pt},
  dvertex/.style={circle,fill=red!80!black,inner sep=1.45pt}]
  \coordinate (ul) at (-2,2);
  \coordinate (um) at (0,2);
  \coordinate (ur) at (2,2);
  \coordinate (ll) at (-2,-2);
  \coordinate (lm) at (0,-2);
  \coordinate (lr) at (2,-2);
  \draw[primal] (ul)--(ll)--(lm)--(lr)--(ur)--(um)--(ul);
  \draw[primal,line width=1.2pt] (um)--(lm);
  \draw[primal] (ul)--(-3.5,3.5) (ll)--(-3.5,-3.5)
                (ur)--(3.5,3.5) (lr)--(3.5,-3.5);
  \foreach \p in {(ul),(um),(ur),(ll),(lm),(lr)}
    \node[pvertex] at \p {};

  \coordinate (A) at (-1,0);
  \coordinate (B) at (1,0);
  \coordinate (C) at (0,4);
  \coordinate (D) at (0,-4);
  \coordinate (X) at (-6,0);
  \coordinate (Y) at (6,0);
  \draw[polar] (A)--(B) (A)--(C)--(B) (A)--(D)--(B)
               (X)--(C)--(Y)--(D)--cycle (X)--(A) (B)--(Y);
  \foreach \p in {(A),(B),(C),(D),(X),(Y)}
    \node[dvertex] at \p {};
  \node[above left=1pt] at (A) {$A$};
  \node[above right=1pt] at (B) {$B$};
  \node[above=2pt] at (C) {$C$};
  \node[below=2pt] at (D) {$D$};
  \node[left=2pt] at (X) {$X$};
  \node[right=2pt] at (Y) {$Y$};
  \node[right=2pt] at (0,0.75) {$e$};
\end{tikzpicture}%
}
\caption{The neighborhood of the deleted edge and the corresponding polar
triangulation. The primal one-skeleton is solid; the polar edges are dashed
and drawn with heavier line weight. The faces \(A\) and \(B\) meet along \(e\), while
\(C,D\) are the third faces at its endpoints. The vertices \(X,Y\) are dual
to the faces opposite \(e\) across \(A,B\), respectively. The two special
polar triangles are \(CAB\) and \(DAB\).}
\label{fig:edge-polar-neighborhood}
\end{figure}

Put
\[
        S_X=\operatorname{ost}(X),\qquad
        S_Y=\operatorname{ost}(Y).
\]
These open stars are disjoint.  Indeed, the relative interiors of distinct
simplices in a triangulation are disjoint.  Hence a point of
\(S_X\cap S_Y\) would lie in the relative interior of one simplex containing
both \(X\) and \(Y\), so that \(XY\) would be a polar edge.  This is exactly
the adjacency excluded above.

\emph{Transported star estimate.}
Let \(Z\) be a cone point of \(M_s\) whose star contains at most one of
\(A,B\); in particular, one may take \(Z=X\) or \(Z=Y\). Let \(\lambda\) be a
local geodesic in \(M_s\), and let
\(\eta\) be the closure of a component of
\(\lambda\cap\operatorname{ost}_{M_s}(Z)\).  If the endpoints of \(\eta\)
lie on \(\partial\operatorname{st}_{M_s}(Z)\), then
\begin{equation}\label{eq:transported-star-estimate}
        \ell(\eta)>\pi .
\end{equation}
Indeed, no triangle in \(\operatorname{st}_{M_s}(Z)\) contains both \(A\)
and \(B\). Deleting \(e\) therefore changes this star, if at all, only by
replacing one of \(A,B\) with the merged vertex \(E\). The cellwise part of \(f\)
maps its triangles bijectively, with the same gluings and the same side lengths
\(L\), onto \(\operatorname{st}_{P_1^\circ}(f(Z))\).  Thus
\[
 f|_{\operatorname{st}_{M_s}(Z)}\colon
 \operatorname{st}_{M_s}(Z)\longrightarrow
 \operatorname{st}_{P_1^\circ}(f(Z))
\]
is an intrinsic isometry. Consequently \(f(\eta)\) is a local geodesic
passage through the corresponding star in \(P_1^\circ\), and
\(\ell(f(\eta))=\ell(\eta)\).  Since \(P_1\) exists by \textup{(S2)} and is
non-obtuse, Lemma~\ref{lem:ordinary-star-passage} applies in
\(P_1^\circ\) and gives \eqref{eq:transported-star-estimate}.  Notice that
the lemma has not been applied to \(M_s\) itself.

For a fixed local carrier the same strict inequality holds for a nonconstant
one-sided limit of complete star passages.  Indeed, after development the
possible endpoints lie in the compact union of the boundary sides of the
star, and every such endpoint has negative scalar product with the developed
centre.  The equality configuration of Lemma~\ref{lem:ordinary-star-passage}
would require antipodal endpoints, which is therefore excluded also in the
limit.  We shall use this observation when a germ runs along an edge of the
triangulation.

Suppose, toward a contradiction, that \(M_s\) contains a closed local geodesic
\(\Gamma\) with \(\ell(\Gamma)\le2\pi\).  We first assume that \(\Gamma\) is
transverse to the polar triangulation.  It cannot be contained in \(K_s\).
Indeed, in the coordinates used in the proof of
Lemma~\ref{lem:reflex-patch-crossing}, the intersection of \(K_s\) with the
great circle \(z=0\) is the short arc \(AB\).  Every other great circle meets
\(z=0\) in two antipodal points, and those two points cannot both belong to
\(AB\), whose length is \(s<\pi\).  Thus \(K_s\) contains no complete great
circle.

If \(\Gamma\) misses \(\operatorname{int}K_s\), then \(f\) is a local
isometry along \(\Gamma\).  Indeed, transversality excludes a segment or a
tangency along \(\partial K_s\), while a transverse crossing of a boundary
side would enter \(\operatorname{int}K_s\).  Thus \(f(\Gamma)\) is a closed
local geodesic in
\(P_1^\circ\).  By \textup{(RH3)} for \(P_1^\circ\) and
\eqref{eq:collapse-map-length},
\[
        \ell(\Gamma)\ge \ell(f(\Gamma))>2\pi,
\]
a contradiction.  Hence
\(\Gamma\cap\operatorname{int}K_s\ne\varnothing\).  Let \(\kappa\) be the
closure of a component of this intersection.  The boundary sides of \(K_s\),
in cyclic order, are
\[
        CA,\ AD,\ DB,\ BC.
\]
The endpoints of \(\kappa\) cannot lie on the same boundary side.  After
development, \(\kappa\) and that side lie on great circles.  Two distinct
great circles meet in an antipodal pair, whereas every boundary side has
length \(L<\pi\); equality of the two great circles is excluded by
transversality.  Thus there are six unordered pairs of boundary sides.  Four
are the nonexceptional pairs
\begin{equation}\label{eq:nonexceptional-boundary-pairs}
 \{CA,CB\},\quad \{AD,BD\},\quad
 \{CA,BD\},\quad \{AD,CB\}.
\end{equation}
Here and below, a \emph{complete passage through the star of \(Z\)} means the
closure of a component of
\(\Gamma\cap\operatorname{ost}(Z)\); its endpoints lie on
\(\partial\operatorname{st}(Z)\).

At an endpoint on \(CA\) or \(AD\), the exterior continuation of \(\Gamma\)
enters respectively the equilateral triangle \(CAX\) or \(DAX\), and hence
\(S_X\).  At an endpoint on \(CB\) or \(BD\), it enters \(S_Y\).
Consequently, if \(\kappa\) has one of the boundary pairs
in \eqref{eq:nonexceptional-boundary-pairs}, the exterior part of \(\Gamma\)
contains one complete passage through \(\operatorname{st}(X)\) and one
through \(\operatorname{st}(Y)\).  Their interiors lie in the disjoint sets
\(S_X,S_Y\), so the passages are disjoint.  The map \(f\) sends each passage
isometrically to a local geodesic segment in the corresponding unchanged
star of \(P_1^\circ\).  Thus the transported star estimate
\eqref{eq:transported-star-estimate} gives
\[
        \ell(\Gamma)>2\pi.
\]

It remains to consider the two exceptional boundary pairs.  Interchanging
\((A,X)\) with \((B,Y)\) if necessary, the chosen component has order
\begin{equation}\label{eq:same-side-component}
        CA\longrightarrow AB\longrightarrow AD.
\end{equation}
By Lemma~\ref{lem:reflex-patch-crossing},
\[
        \ell(\kappa)>\pi .
\]
Follow \(\Gamma\) from either endpoint of \(\kappa\) in the direction away
from \(K_s\).  It first enters \(CAX\) or \(DAX\), hence \(S_X\), and before
it can return to \(K_s\) it completes a passage through the unchanged closed
star of \(X\).  Its image under \(f\) is a local geodesic passage through the
corresponding star in \(P_1^\circ\), so
\eqref{eq:transported-star-estimate} gives length greater than \(\pi\) for this
passage.
It is disjoint from the interior of \(\kappa\), so
\[
        \ell(\Gamma)\ge \ell(\kappa)+\pi>2\pi,
\]
a contradiction.  Notice that this argument uses only the chosen component
\(\kappa\); it is immaterial whether \(\Gamma\cap\operatorname{int}K_s\) has
other components.  This is the same economy used in Inoue's proof.

It remains to remove the transversality assumption.  We use the same list of
degenerations as Inoue.  Only one additional local convention is needed for
our non-bigon patch.  Suppose that a passage in \(K_s\), not passing through
\(A\) or \(B\) as an interior point, is obtained from a transverse passage by
letting an intersection point reach an endpoint of a side or by letting the
supporting great circle approach a boundary side.  Develop the one or two
special triangles met by the passage.  Moving the developed great circle a
small amount from the chosen incident side gives transverse passages with the
same ordered sides; moving it back gives precisely a one-sided limit in the
sense defined before Lemma~\ref{lem:reflex-patch-crossing}.  This is a local
statement about the passage in \(K_s\), not a perturbation of the whole closed
geodesic.  Passages through \(A\) or \(B\) are handled directly by the vertex
part of that lemma.

Every portion of the geodesic in an unchanged open star is the closure of an
actual component of \(\Gamma\cap\operatorname{ost}(Z)\), so the transported
star estimate applies to it directly, without a limiting argument.  No global
perturbation of \(\Gamma\), and in particular no globally closed perturbed
geodesic, is asserted anywhere below.

We now list the carrier types.  This is exactly Inoue's five-part division
relative to the common edge, with the two opposite vertices \(C,D\) treated
separately because our patch is not a bigon.  At an interior point of \(AB\) the
metric is smooth.  A geodesic which is tangent to the great circle containing
\(AB\) coincides with it locally, by uniqueness of a spherical geodesic with
prescribed initial data.  Hence an isolated tangency to
\(\operatorname{int}AB\) is impossible: the passage either crosses
\(\operatorname{int}AB\) transversely or contains a nontrivial segment of
\(AB\).  In the latter case the same uniqueness forces it to follow \(AB\)
up to both endpoints.
In cases \textup{(i)}--\textup{(iv)} below the carrier avoids the two opposite
vertices \(C,D\); carriers through either of them are precisely case
\textup{(v)}.

\begin{enumerate}[label=\textup{(\roman*)},itemsep=0.35em,topsep=0.4em]
\item If the carrier misses \(AB\), its part in \(K_s\) lies in one special
triangle, with a boundary edge or vertex allowed as a one-sided limit.  The
cyclic orders of the incident triangles show that its two exterior
continuations determine complete passages through the unchanged stars of
\(X\) and \(Y\).

\item Suppose that the carrier crosses \(\operatorname{int}AB\)
transversely.  This repeats the local configuration from the transverse proof,
not its global hypothesis: the complete closed geodesic may still run along an
outer edge or pass through a vertex of the triangulation.  The outer sides form
one of
\[
 \{CA,AD\},\quad \{CB,BD\},\quad
 \{CA,BD\},\quad \{CB,AD\}.
\]
For the last two pairs the exterior continuations give complete passages
through both unchanged stars.  For the first two pairs the one-sided-limit
part of Lemma~\ref{lem:reflex-patch-crossing} gives a patch passage of length
greater than \(\pi\), and the exterior continuation gives a complete passage
through \(X\), respectively \(Y\), of length greater than \(\pi\).

\item Suppose that the carrier contains a nontrivial segment of \(AB\).
As observed above, it then contains all of \(AB\).  At either endpoint the
sector from \(AB\) into one special triangle has angle
\(\delta<\pi\), whereas the sector continuing into the adjacent ordinary
triangle has angle \(\delta+\beta>\pi\).  Local geodesicity therefore forces
the two continuations into the unchanged stars of \(X\) and \(Y\).

\item Suppose that the carrier meets \(AB\) only at \(A\) or only at \(B\).
There is no germ along \(AB\), since that would be case \textup{(iii)}.
If exactly one germ lies in \(K_s\), follow that germ through its special
triangle to the outer boundary of the patch. At \(A\), the exterior germ
lies in the open star of \(X\), whereas the patch germ exits through
\(BC\) or \(BD\) into the open star of \(Y\). The case at \(B\) is
symmetric. The excluded exits through \(C,D\) belong to case \textup{(v)}.
Thus the two complete passages lie in the disjoint open stars of \(X\)
and \(Y\), so their parameter intervals are disjoint and each has length
greater than \(\pi\).

If both germs stay in \(K_s\), the vertex part of
Lemma~\ref{lem:reflex-patch-crossing} gives a patch passage of length greater
than \(\pi\); its two outer ends supply one complete passage through the
opposite unchanged star.

If both germs are exterior, the geodesic only touches \(K_s\) at the vertex.
This contact is irrelevant if \(\Gamma\) has another carrier through the
patch.  If there is no such carrier, then all intersections with \(K_s\) are
isolated exterior vertex contacts of this type; a contact along a boundary
edge has already been assigned to a one-sided carrier.
The map \(f\) preserves the exterior sector, whose total angle is \(2\beta\).
If \(\theta\) is the angle between the two germs measured inside this sector,
local geodesicity gives \(\theta\ge\pi\).  After \(A,B\) are identified with
\(E\), the other angle between their images is \(4\beta-\theta\ge2\beta>
\pi\).  Applying this at every contact shows that \(f(\Gamma)\) is still
locally geodesic, and
\textup{(RH3)} for \(P_1^\circ\), together with the \(1\)-Lipschitz property
of \(f\), gives \(\ell(\Gamma)>2\pi\) directly.

\item Finally, suppose that the carrier passes through \(C\) or \(D\); take
\(C\) for definiteness. Two germs contained in \(K_s\) subtend the special
angle \(\gamma<\pi\), so they cannot form a local geodesic. Hence at least one
germ is exterior.

Suppose first that exactly one germ is exterior.  If it follows \(CA\) or
\(CB\), it reaches \(A\) or \(B\) and is already covered by case
\textup{(iv)}.  Otherwise the ordinary triangle containing it has a vertex
\(Z\notin\{A,B,C,D\}\), and the adjacent portion of
\(\Gamma\) is the closure of a component of
\(\Gamma\cap\operatorname{ost}(Z)\) ending at \(C\).  The transported star
estimate gives length greater than \(\pi\) for its complete passage through
\(\operatorname{st}(Z)\).  The other germ leaves \(ABC\) through the opposite
side \(AB\). If it reaches \(A\) or \(B\), this is case \textup{(iv)}.
Otherwise it crosses \(\operatorname{int}AB\) and leaves the other special
triangle through \(AD\) or \(BD\). It is therefore a nonconstant one-sided
limit of a same-side passage in
Lemma~\ref{lem:reflex-patch-crossing}, and its length is greater than \(\pi\).
The interior of this patch passage lies in \(K_s\), whereas
\(\operatorname{ost}(Z)\) contains no special triangle, so the two selected
passages have disjoint interiors.

It remains to treat two exterior germs.  This is the point at which merely
choosing an arbitrary star for each germ would be insufficient: the resulting
complete passages could overlap later on the parameter circle.  Instead we use
the effect of the collapse on the link of \(C\).  In \(P_1^\circ\), the
exterior fan at \(C\) consists of \(m\ge4\) ordinary equilateral triangles,
each contributing the angle \(\beta\); the inequality \(m\ge4\) follows from
the same triangular-face argument used in the proof of
Theorem~\ref{thm:atkinson-deformation}.  Its two end triangles are \(CBY\) and
\(CAX\).  Let \(a,b\ge0\) be the angular distances, within this exterior fan,
from the two germs to the boundary rays \(CB\) and \(CA\), respectively.  The
two link angles between the germs in \(M_s\) are
\begin{equation}\label{eq:two-exterior-link-angles}
        m\beta-a-b,
        \qquad
        a+b+\gamma,
\end{equation}
where the second angle is the one containing the inserted special sector.
Both are at least \(\pi\), by local geodesicity of \(\Gamma\).

The collapse removes precisely the angle \(\gamma\).  Thus the corresponding
angles between the image germs in \(P_1^\circ\) are
\[
        m\beta-a-b,
        \qquad a+b.
\]
If \(a+b\ge\pi\), then \(f(\Gamma)\) remains locally geodesic at this contact.
If \(a+b<\pi\), then
\[
        a<\pi<2\beta,
        \qquad b<\pi<2\beta.
\]
If \(a=0\), the corresponding germ follows the geodesic edge \(CB\) up to
\(B\), by uniqueness of a spherical geodesic with prescribed initial data,
and the carrier is already case \textup{(iv)}.  Similarly, \(b=0\) leads to
case \textup{(iv)} at \(A\).  We may therefore assume \(a,b>0\).
The first two triangles of the exterior fan starting at \(CB\) are both
contained in \(\operatorname{st}(Y)\), and the first two starting at \(CA\)
are both contained in \(\operatorname{st}(X)\).  Hence the two germs determine
actual complete passages through \(S_Y\) and \(S_X\), respectively.  These
passages have length greater than \(\pi\), and their interiors are disjoint
because \(S_X\cap S_Y=\varnothing\).  This gives
\(\ell(\Gamma)>2\pi\).

The same argument applies at \(D\).  The one-sided convention preceding
Lemma~\ref{lem:reflex-patch-crossing} handles the remaining coincidences of a
carrier with an edge of the triangulation.
\end{enumerate}

Whenever the argument in cases \textup{(i)}--\textup{(iv)} uses two star
passages, their interiors lie in the disjoint open stars \(S_X,S_Y\), except
in the one-germ subcase of \textup{(iv)}.  In that subcase they are distinct
parameter components separated by the patch segment, so their interiors are
again disjoint.  Whenever the argument uses a patch passage and a star
passage, their interiors lie respectively in \(K_s\) and in an unchanged
open star, and are therefore disjoint.

It remains only to assemble the last subcase of \textup{(v)} globally.  If at
some two-exterior contact one has \(a+b<\pi\), the two disjoint \(X\)- and
\(Y\)-passages found above already give \(\ell(\Gamma)>2\pi\).  Otherwise
\(a+b\ge\pi\) at every such contact.  If none of the earlier cases has already
given the required contradiction, \(\Gamma\) meets \(K_s\) only in these
exterior contacts, and the link calculation shows that \(f(\Gamma)\) is a
closed local geodesic in \(P_1^\circ\).  Condition \textup{(RH3)} there and
the \(1\)-Lipschitz property of \(f\) give
\[
        \ell(\Gamma)\ge\ell(f(\Gamma))>2\pi.
\]
Thus every carrier type leads to a contradiction, proving \textup{(RH3)}.

Theorem~\ref{thm:rhs} now realizes \(M_s\) as the polar metric of a unique
compact convex hyperbolic polyhedron; denote it by \(Q_s\).  There is one
point which must still be checked.  The metric \(M_s\) was built with a
prescribed triangulation, but Theorem~\ref{thm:rhs} by itself does not say
that this triangulation is the actual polar cellulation of \(Q_s\).

We identify the two cellulations exactly as in the last part of Inoue's proof
of geometric edge surgery \cite[proof of Theorem~8.2]{Inoue2008}. Let
\(\mathcal C\) denote the triangulation used to construct \(M_s\), and let
\(\mathcal D\) be the actual polar cellulation of \(Q_s\). The latter has the
following three properties:
\begin{enumerate}[label=\textup{(C\arabic*)},itemsep=0.15em,topsep=0.3em]
\item its vertices are exactly the cone points of \(M_s\);
\item its edges are geodesic arcs of length strictly less than \(\pi\);
\item every interior angle of a two-cell at a cone point is strictly less than
\(\pi\).
\end{enumerate}
The first property is part of the polar construction. The second follows
because a polar edge has length \(\pi-\theta\), where
\(0<\theta<\pi\) is a dihedral angle of the convex polyhedron. The third is
the polar form of the strict convexity of its faces: a polar face angle is the
complement of the corresponding face angle of the primal polyhedron. The
prescribed triangulation \(\mathcal C\) also has these properties, since its
edge lengths are \(s,L<\pi\) and its face angles are
\(\beta,\delta(s),\gamma(s)\in(0,\pi)\).

We next record the metric observation which replaces the right-angled
dichotomy used by Inoue. In the known polar \(P_1^\circ\), two distinct cone
points are either joined by a polar edge and have distance \(L\), or their
distance is strictly greater than \(\pi\). To prove the second assertion, let
\(\sigma\) be a shortest path between nonadjacent cone points \(U,V\). Follow
\(\sigma\) from \(U\) up to its first intersection with
\(\partial\operatorname{st}(U)\). This initial subpath is minimizing and lies
in \(\operatorname{st}(U)\). Develop the triangles which it traverses while
keeping the developed image of \(U\) fixed. Its endpoint lies on a side
opposite \(U\) in an equilateral spherical triangle of side length
\(L>\pi/2\). The scalar-product argument in the proof of
Lemma~\ref{lem:ordinary-star-passage} therefore shows that this initial
subpath has length strictly greater than \(\pi/2\). Applying the same argument
backwards from \(V\) gives a terminal subpath of length strictly greater than
\(\pi/2\). Since \(U,V\) are nonadjacent, their open stars are disjoint, so
the interiors of these two subpaths are disjoint. Hence
\(\ell(\sigma)>\pi\). Notice that this is an estimate for the developed
initial and terminal subpaths of \(\sigma\), not an assertion about a
possibly different global shortest path to a boundary point of a star.

On the other hand, \(P_1^\circ\) is a
\(\operatorname{CAT}(1)\) space: it is locally \(\operatorname{CAT}(1)\), and
\textup{(RH3)} rules out closed local geodesics of
length at most \(2\pi\). Hence every local geodesic of length less than
\(\pi\) is the unique minimizing segment between its endpoints. In
particular, for adjacent cone points the polar edge is minimizing and their
distance is exactly \(L\).

Now let \(U,V\) be cone points of \(M_s\) with
\(d_{M_s}(U,V)<\pi\). Since \(f\) is \(1\)-Lipschitz,
\[
 d_{P_1^\circ}(f(U),f(V))
 \le d_{M_s}(U,V)<\pi.
\]
Thus either \(f(U)=f(V)\), or \(f(U)\) and \(f(V)\) are adjacent in
\(P_1^\circ\). The only two cone points identified by \(f\) are \(A,B\), and
they are joined by the edge \(AB\) of \(\mathcal C\). Every edge of
\(P_1^\circ\) not incident to \(E=f(A)=f(B)\) has one lift, which is an edge
of \(\mathcal C\). The four edges incident to \(E\) have endpoints
\(C,D,X,Y\). Their lifts are
\[
 AC,BC,AD,BD,AX,BY,
\]
together with the two possible additional pairs
\begin{equation}\label{eq:possible-extra-polar-edges}
        AY,\qquad BX.
\end{equation}
The metric \(M_s\) itself is now also known to be
\(\operatorname{CAT}(1)\), by \textup{(RH1)}--\textup{(RH3)}. Hence a
geodesic segment of length less than \(\pi\) with prescribed endpoints is
unique. Whenever the endpoint pair above is joined by an edge of
\(\mathcal C\), an edge of \(\mathcal D\) with the same endpoints must
therefore be that very geodesic arc.
It follows that an edge of \(\mathcal D\) is either an edge of
\(\mathcal C\), or one of the two arcs in
\eqref{eq:possible-extra-polar-edges}. This is precisely the role of the map
\(f\) in Inoue's argument: the distance classification is made in the known
polar \(P_1^\circ\), not in the candidate metric \(M_s\), and is then pulled
back by a length-nonincreasing map.

We now recover the edges of \(\mathcal C\), in the same order as in Inoue's
link argument. Suppose that a prescribed edge is absent from \(\mathcal D\).
Unless another \(\mathcal D\)-edge enters the sector between the two
prescribed triangles incident to it, those triangles belong to one
\(\mathcal D\)-face. At a suitable endpoint that face has angle at least
\[
\begin{array}{c|c}
\text{missing prescribed edge}&\text{resulting face angle}\\
\hline
\text{an unchanged edge}&2\beta,\\
AB&2\delta(s),\\
AC,AD,BC,\text{ or }BD&\beta+\delta(s).
\end{array}
\]
All three lower bounds are strictly greater than \(\pi\), by
\eqref{eq:delta-gamma-bounds} and \eqref{eq:L-beta-ranges}, contradicting
\textup{(C3)}.

It remains only to check that an exceptional arc from
\eqref{eq:possible-extra-polar-edges} cannot evade this conclusion. We first
localize such an arc; this is the extra step which is unnecessary in the
right-angled bigon of Inoue. Suppose that \(\sigma=AY\) is an edge of
\(\mathcal D\). At \(A\), the four sectors of \(\mathcal C\) are the
triangles
\[
        ABC,\qquad ACX,\qquad AXD,\qquad ADB.
\]
The arc \(\sigma\) cannot initially follow an edge of \(\mathcal C\): two
local geodesics of length less than \(\pi\) with the same initial data agree,
and then \(\sigma\) would either have the wrong endpoint or contain a third
cone point in its interior. If \(\sigma\) leaves through \(ACX\), it must
cross the opposite side \(CX\) before reaching \(Y\). The crossing is in the
interior of \(CX\), for an edge of a cell decomposition cannot contain a
third cone point. If \(CX\in\mathcal D\), two \(\mathcal D\)-edges cross,
which is impossible. If \(CX\notin\mathcal D\), the two ordinary triangles
incident to \(CX\) lie in one \(\mathcal D\)-face near \(C\), unless a new
\(\mathcal D\)-edge issues from \(C\) into that sector. No such edge exists:
every possible non-prescribed edge is \(AY\) or \(BX\), and neither is
incident to \(C\). The resulting face angle at \(C\) is therefore at least
\(2\beta>\pi\), contrary to \textup{(C3)}. The sector \(AXD\) is excluded in
the same way, using the endpoint \(D\) of the opposite side \(XD\).

Thus \(AY\) leaves \(A\) through \(ABC\) or \(ABD\). In the first case it
crosses \(BC\) and enters \(BCY\); in the second it crosses \(BD\) and enters
\(BDY\). The subarc after this crossing is the unique geodesic segment of
length less than \(\pi\) to \(Y\), so it remains in the corresponding strictly
convex spherical triangle. Hence \(AY\) is one of the two local flip arcs
crossing \(BC\) or \(BD\). Symmetrically, an exceptional edge \(BX\) is a
local flip arc crossing \(AC\) or \(AD\).

Now suppose, for example, that \(AY\) crosses \(BC\). Then \(BC\) is absent
from \(\mathcal D\), and at \(B\) the merged sector has angle
\(\delta(s)+\beta>\pi\). The only possible non-prescribed edge which can
subdivide it is \(BX\). Its germ at \(B\) must lie in the triangle \(ABC\),
so the localization just proved shows that it crosses \(AC\). Both arcs then
lie in the topological disc
\[
        ABC\cup ACX\cup BCY,
\]
whose boundary contains \(A,X,Y,B\) in this cyclic order. The arcs \(AY\)
and \(BX\) therefore have alternating endpoints and must cross in their
interiors, impossible for two edges of \(\mathcal D\). The three remaining
possibilities are identical, using the corresponding upper or lower
three-triangle disc. Consequently an exceptional edge cannot subdivide every
forbidden sector created by an omitted prescribed edge. Some
\(\mathcal D\)-face would then have one of the three angles displayed above,
contrary to \textup{(C3)}. Hence every edge of \(\mathcal C\) is an edge of
\(\mathcal D\).

The only possible additional edges of \(\mathcal D\) are therefore \(AY\)
and \(BX\). Either such arc joins two nonadjacent vertices of the triangulation
\(\mathcal C\), and hence meets its \(1\)-skeleton away from its endpoints.
It cannot pass through a third cone point in the interior of an edge of
\(\mathcal D\); therefore it crosses an edge of \(\mathcal C\) in an
interior point. Since all edges of \(\mathcal C\) have just been shown to belong to
\(\mathcal D\), this is impossible for a cell decomposition. Thus
\[
        \mathcal C=\mathcal D.
\]
This is the same final combinatorial argument as in Inoue; the only changes
are that the ordinary polar edge length is \(L\), rather than \(\pi/2\), and
that the distance dichotomy is transferred through the map \(f\).

Therefore \(Q_s\) has the prescribed combinatorial type, its angle at \(e\)
is \(t=\pi-s\), and all its other angles are \(\alpha\). For fixed
combinatorics, the metrics \(M_s\) depend smoothly on \(s>0\), and, exactly
as in Inoue's proof of geometric edge surgery
\cite[proof of Theorem~8.2]{Inoue2008}, the realizations \(Q_s\) form a
one-parameter family. At \(s=L\), all angles of \(Q_L\) equal \(\alpha\), and
the recovered cellulation is that of \(P_0\); hence Andreev's uniqueness
gives \(Q_L=P_0\) up to isometry. Set
\[
        P_t=Q_{\pi-t},\qquad \alpha\le t<\pi.
\]
As \(s\to0\), the polar metric \(M_s\) converges to \(P_1^\circ\): the edge
\(AB\) contracts, \(A\) and \(B\) coalesce to \(E\), and the two special
triangles collapse to the path \(CED\). This convergence follows from the
finite spherical gluing construction, with all other triangles unchanged.
To pass from polar metrics to polyhedra while the cone points merge, apply
the compactness lemma \cite[Lemma~6.3]{SchlenkerHyperideal} with the underlying
manifold a ball. The \(Q_s\), for \(s>0\), have fixed combinatorics and no
ideal vertices, and their third fundamental forms are precisely \(M_s\).
The limiting spherical cone metric \(P_1^\circ\) has every closed local
geodesic longer than \(2\pi\). Thus the lemma gives subsequential convergence
up to isometry. The limit is compact: an ideal vertex would contribute a
polar hemisphere with geodesic boundary of length \(2\pi\).
Hodgson--Rivin uniqueness identifies this limit with \(P_1\), so
\(Q_s\to P_1\) up to isometry as \(s\to0\). After normalization, the
polyhedra lie in a compact subset of the Klein ball and converge there in
Hausdorff distance. The hyperbolic volume density is smooth on that subset,
and hence \(\Vol(Q_s)\to\Vol(P_1)\).

Finally \(t=\pi-s\), and the Schl\"afli formula gives
\[
        \frac{d}{dt}\Vol(P_t)=-\frac12\ell_e(t)<0.
\]
Integration from \(t=\alpha\) to the limiting value \(t=\pi\) proves
\(\Vol(P_0)>\Vol(P_1)\).
\end{proof}

\subsection{Successive removal of leaves}\label{subsec:graph-collapse}

Choose any leaf of the current prismatic tree and shorten its free chain by
successive seam deletions. Once the leaf has become a \(4\)-prism, it is
absorbed combinatorially by the adjacent prism. We then recompute the reduced
prismatic tree of the resulting graph-type polyhedron and repeat the operation.

\begin{lemma}[Removal of one leaf]\label{lem:leaf-collapse}
Let the current reduced prismatic tree have more than one vertex and let
\(D\) be one of its leaves.  Then \(D\) is an \(n\)-prism, \(n\ge5\), and
\(n-4\) direct equiangular edge surgeries reduce it to a \(4\)-prism, which is
then absorbed combinatorially by the adjacent prism. Every surgery
satisfies the hypotheses of
Lemma~\ref{lem:direct-equiangular-edge-surgery}, leaves all surviving angles
equal to \(\alpha\), and strictly decreases volume.
\end{lemma}

\begin{proof}
The inequality \(n\ge5\) is
Theorem~\ref{thm:graph-type-structure}.  By \eqref{eq:leaf-free-chain}, the
free side faces form a linear chain
\[
        F_1,F_2,\ldots,F_{n-3}.
\]
These are the actual boundary side faces of \(D\): the \(n-3\) free
quadrilateral faces outside the three-position attaching collar.  In
particular, none of them is a base or an internal cutting face.
Starting with \(H_1=F_1\), define successively
\[
 e_i=H_i\cap F_{i+1},\qquad
 H_{i+1}=H_i\cup_{e_i}F_{i+1},\qquad
 1\le i\le n-4.
\]
After the first \(i-1\) deletions, \(H_i\) is a quadrilateral: its two chains
in the base directions have concatenated to two edges, and its two end edges
are unchanged.  Thus the two faces incident to the active seam \(e_i\), namely
\(H_i\) and \(F_{i+1}\), are quadrilaterals.

Delete \(e_i\) abstractly.  Combinatorially, this merges two consecutive
side faces of the leaf and concatenates the corresponding pairs of edges in
its two bases.  Hence it replaces the current \((n-i+1)\)-prism leaf by an
\((n-i)\)-prism leaf, without changing the attaching face or any other
block.  Since \(1\le i\le n-4\), the new leaf has at least four side faces.
Thus the post-deletion polyhedron is again a tree of prisms, all with at
least four side faces.  Lemma~\ref{lem:prism-tree-no-three} shows directly
that it has no prismatic \(3\)-circuit.

The post-deletion abstract polyhedron is still trivalent, and
\(3\alpha>\pi\).  Every prismatic \(4\)-circuit has angle sum
\(4\alpha<2\pi\).  The absence of prismatic \(3\)-circuits also excludes a
triangular prism.  Moreover the strict inequality
\(3\alpha>\pi\) makes every trivalent vertex finite, so the clause of
Andreev's theorem involving an ideal vertex does not arise.
Consequently Theorem~\ref{thm:andreev} realizes it as a compact equiangular
hyperbolic polyhedron.  Lemma~\ref{lem:circuit-correspondence} now shows that
the active seam \(e_i\) belonged to no prismatic \(4\)-circuit before its
deletion.  Hence all three hypotheses of
Lemma~\ref{lem:direct-equiangular-edge-surgery} hold.  Increasing only
\(\theta_{e_i}\) from \(\alpha\) to \(\pi\) realizes the deletion, strictly
decreases volume, and leaves every surviving edge at angle \(\alpha\).

After the \(n-4\) surgeries, all \(n-3\) free faces have merged to one
quadrilateral, and the leaf block has become a \(4\)-prism. It is absorbed by
the neighboring prism.

The remaining displayed blocks retain their prism product structures, and
every remaining gluing is still crossed.  They therefore give an explicit
graph-orbifold decomposition entirely into prism pieces.  In particular, the
resulting abstract polyhedron is again of graph type.  We do not claim that
this displayed decomposition is already reduced: edge
deletion can change the full collection of prismatic \(4\)-circuits.  At the
next step we simply take the reduced prismatic tree of the new graph-type
polyhedron and apply Theorem~\ref{thm:graph-type-structure} afresh.

\end{proof}

\subsection{Reduction to an ordinary prism}

\begin{proposition}[Graph-type reduction]\label{prop:graph-surgery-reduction}
Let \(P_\alpha\) have graph type.  If it is not already an ordinary prism, then
for some \(n\ge4\),
\[
        \Vol(P_\alpha)>\Vol(\Pi_n(\alpha)).
\]
\end{proposition}

\begin{proof}
At the current stage take the reduced prismatic tree.  If it has more than
one vertex, choose a leaf and apply Lemma~\ref{lem:leaf-collapse}; after its
free chain is shortened, absorb the resulting \(4\)-prism and take the
prismatic tree of the new graph-type polyhedron. Each surgery deletes a boundary seam,
so the number of faces strictly decreases.  The process is therefore finite,
and it stops when the current polyhedron is an ordinary equiangular
\(n\)-prism, \(n\ge4\).  Since the original polyhedron was not already an
ordinary prism, at least one direct surgery occurs.  Every such surgery
strictly decreases volume, which gives the stated strict inequality.
\end{proof}

\subsection{Ordinary equiangular prisms}

Let \(\Pi_n(\alpha)\) be an ordinary equiangular \(n\)-prism. A triangular
prism is impossible in our range, since its three lateral edges form a
prismatic \(3\)-circuit and Andreev's condition would require
\(3\alpha<\pi\). Hence \(n\ge4\).

Uniqueness in Andreev's theorem realizes the combinatorial dihedral
symmetries and the interchange of the bases by isometries. Centering their
common fixed point in the Klein ball gives the symmetric model below.

The volume decreases when the common dihedral angle increases. It is therefore
enough to estimate the limiting prism at
\(\alpha_0=\arccos(1/3)\). Put \(c=\cos\alpha_0=1/3\). In the Klein ball,
place the bases in the planes \(z=\pm h\), and let the side planes be vertical
and tangent to a Euclidean circle of radius \(\rho\). If
\(q=\cos(2\pi/n)\), the Klein angle formulas give
\begin{equation}\label{eq:ordinary17-rho}
        \rho^2=\frac{c+q}{1+c},
\end{equation}
and
\begin{equation}\label{eq:ordinary17-h}
        h=\frac{c\sqrt{1-\rho^2}}
        {\sqrt{\rho^2+c^2(1-\rho^2)}}.
\end{equation}
The base polygon contains the Euclidean disk of radius \(\rho\). Therefore
\begin{align}
        \Vol(\Pi_n(\alpha_0))
        &\ge
        \int_{-h}^{h}\int_0^{2\pi}\int_0^\rho
        \frac{r\,dr\,d\theta\,dz}{(1-r^2)^2}\nonumber\\
        &=2h\,\frac{\pi\rho^2}{1-\rho^2}.
        \label{eq:ordinary17-cylinder}
\end{align}
Here the hyperbolic volume density in the Klein ball is
\((1-r^2-z^2)^{-2}\).  We deliberately replaced it by the smaller quantity
\((1-r^2)^{-2}\), which gives the displayed lower bound after integration.
Writing \(x=\rho^2\), the right-hand side is
\[
        F(x)=2\pi c\,
        \frac{x}{\sqrt{1-x}\sqrt{x+c^2(1-x)}}.
\]
For \(n=4\), equations \eqref{eq:ordinary17-rho} and
\eqref{eq:ordinary17-h} give \(x=1/4\) and \(h=1/2\), so
\[
        \Vol(\Pi_4(\alpha_0))\ge\frac\pi3>\vtet.
\]
Moreover,
\begin{align*}
        \frac{F'(x)}{F(x)}
        &=\frac1x+\frac{1}{2(1-x)}
        -\frac{1-c^2}{2\bigl(c^2+(1-c^2)x\bigr)}
        \\
        &=\frac{c^2(2-x)+(1-c^2)x}
        {2x(1-x)\bigl(c^2+(1-c^2)x\bigr)}>0
        \qquad (x\ge1/4,\ c=1/3).
\end{align*}
Since \(x\) increases with \(n\), the same lower bound holds for every
\(n\ge4\).

\begin{lemma}[Ordinary prisms]\label{lem:ordinary-prisms-v17}
For every ordinary equiangular prism \(\Pi_n(\alpha)\), \(n\ge4\), with
\(\pi/3<\alpha<\alpha_0=\arccos(1/3)\), one has
\[
        \Vol(\Pi_n(\alpha))>\vtet>\Vol(T_\alpha).
\]
\end{lemma}

\begin{proof}
The preceding calculation gives
\(\Vol(\Pi_n(\alpha_0))\ge\pi/3>\vtet\). By Schl\"afli's formula,
lowering the common angle from \(\alpha_0\) to \(\alpha\) increases volume.
Finally, \(\Vol(T_\alpha)<\vtet\) for \(\alpha>\pi/3\).
\end{proof}

\begin{lemma}[Prismatic case]\label{lem:prismatic-estimate-v13}
Let \(P_\alpha\) be an ordinary equiangular prism or an equiangular graph-type
polyhedron, with \(\pi/3<\alpha<\arccos(1/3)\). Then
\[
        \Vol(P_\alpha)>\Vol(T_\alpha).
\]
\end{lemma}

\begin{proof}
If \(P_\alpha\) is an ordinary equiangular prism, the result follows from
Lemma~\ref{lem:ordinary-prisms-v17}.

If \(P_\alpha\) has graph type and is not an ordinary prism, combine
Proposition~\ref{prop:graph-surgery-reduction} with
Lemma~\ref{lem:ordinary-prisms-v17}.
\end{proof}

\section{Proof of the main theorem}

\begin{proof}[Proof of Theorem~\ref{thm:main}]
For \(\alpha=\pi/3\), let \(N\) be the number of vertices.  If \(N=4\), the
polyhedron is the ideal regular tetrahedron.  If \(N>4\), Atkinson's estimate
for \(\pi/3\)-equiangular hyperbolic polyhedra gives
\(\Vol(P)\ge N\vtet/3>\vtet\)
\cite[Theorem~2.6]{AtkinsonEquiangular}.  Hence assume that
\[
        \frac{\pi}{3}<\alpha<\arccos\frac13.
\]
Let \(P_\alpha\) be an arbitrary equiangular hyperbolic polyhedron with this angle. If \(P_\alpha\) is a tetrahedron, then the Gram-matrix criterion for hyperbolic simplices, together with the equality of all six dihedral angles, shows that it is isometric to the regular tetrahedron \(T_\alpha\).

Suppose now that \(P_\alpha\) is not a tetrahedron. Consider Atkinson's decomposition. If it has an atoroidal component whose right-angled limit is not isometric to \(P(3,2)\), then by Lemma~\ref{lem:atoroidal-general}
\[
        \Vol(P_\alpha)>\Vol(T_\alpha).
\]

If the right-angled limit of an atoroidal component is \(P(3,2)\), then, by Lemma~\ref{lem:Bhat-volume}, this component already has volume greater than \(\Vol(T_\alpha)\). Hence the whole polyhedron also has volume greater than \(\Vol(T_\alpha)\).

It remains to consider the case where there are no atoroidal components. Then
\(P_\alpha\) has graph type; this includes an ordinary prism as the one-block
case. By Lemma~\ref{lem:prismatic-estimate-v13},
\[
        \Vol(P_\alpha)>\Vol(T_\alpha).
\]

Thus every equiangular polyhedron different from the tetrahedron has volume greater than \(\Vol(T_\alpha)\). The theorem follows.
\end{proof}

\subsection*{Funding}

This work was supported by the state task of the Sobolev Institute of
Mathematics, project No.~FWNF-2026-0011. 

\subsection*{Use of AI tools}
The author used OpenAI's
ChatGPT as an editorial and verification aid when checking the
exposition, proof structure, and bibliographic references. The author takes
full responsibility for the final text, arguments, and references.

\subsection*{Competing interests}
The author has no relevant financial or non-financial interests to disclose.


\begin{thebibliography}{99}

\bibitem{AndreevCompact}
E. M. Andreev,
Convex polyhedra in Loba\v{c}evski\u{\i} spaces,
Math. USSR Sb. \textbf{10} (1970), 413--440.
\url{https://doi.org/10.1070/SM1970v010n03ABEH001677}.

\bibitem{AndreevFinite}
E. M. Andreev,
On convex polyhedra of finite volume in Loba\v{c}evski\u{\i} space,
Math. USSR Sb. \textbf{12} (1970), 255--259.
\url{https://doi.org/10.1070/SM1970v012n02ABEH000920}.

\bibitem{AtkinsonEquiangular}
C. K. Atkinson,
Volume estimates for equiangular hyperbolic Coxeter polyhedra,
Algebr. Geom. Topol. \textbf{9} (2009), no.~2, 1225--1254.
\url{https://doi.org/10.2140/agt.2009.9.1225}.

\bibitem{Atkinson2011}
C. K. Atkinson,
Two-sided combinatorial volume bounds for non-obtuse hyperbolic polyhedra,
Geom. Dedicata \textbf{153} (2011), 177--211.
\url{https://doi.org/10.1007/s10711-010-9563-y}.

\bibitem{AtkinsonRafalski}
C. K. Atkinson, S. Rafalski,
The smallest Haken hyperbolic polyhedra,
Proc. Amer. Math. Soc. \textbf{141} (2013), no. 4, 1393--1404.
\url{https://doi.org/10.1090/S0002-9939-2012-11665-X}.

\bibitem{CharneyDavis1993}
R. Charney, M. W. Davis,
Singular metrics of nonpositive curvature on branched covers of Riemannian
manifolds,
Amer. J. Math. \textbf{115} (1993), no. 5, 929--1009.
\url{https://doi.org/10.2307/2375063}.

\bibitem{EgorovFullerene}
A. Egorov, A. Vesnin,
On correlation of hyperbolic volumes of fullerenes with their properties,
Comput. Math. Biophys. \textbf{8} (2020), 150--167.
\url{https://doi.org/10.1515/cmb-2020-0108}.

\bibitem{RivinHodgson1993}
C. D. Hodgson, I. Rivin,
A characterization of compact convex polyhedra in hyperbolic \(3\)-space,
Invent. Math. \textbf{111} (1993), 77--111;
corrigendum, Invent. Math. \textbf{117} (1994), 359.
\url{https://doi.org/10.1007/BF01231281}.

\bibitem{HodgsonRivinSmith}
C. D. Hodgson, I. Rivin, W. D. Smith,
A characterization of convex hyperbolic polyhedra and of convex polyhedra inscribed in the sphere,
Bull. Amer. Math. Soc. (N.S.) \textbf{27} (1992), no.~2, 246--251.
\url{https://doi.org/10.1090/S0273-0979-1992-00303-8}.

\bibitem{Inoue2008}
T. Inoue,
Organizing volumes of right-angled hyperbolic polyhedra,
Algebr. Geom. Topol. \textbf{8} (2008), no.~3, 1523--1565.
\url{https://doi.org/10.2140/agt.2008.8.1523}.

\bibitem{Kellerhals1989}
R. Kellerhals,
On the volume of hyperbolic polyhedra,
Math. Ann. \textbf{285} (1989), 541--569.
\url{https://doi.org/10.1007/BF01452047}.

\bibitem{Nonaka2024}
J. Nonaka, H. Yoshida,
Volumes and arithmeticity of \(\pi/3\)-equiangular hyperbolic polyhedra,
Kodai Math. J. \textbf{49} (2026), no.~2, 141--160.
\url{https://doi.org/10.2996/kmj49202}.

\bibitem{RHD}
R. K. W. Roeder, J. H. Hubbard, W. D. Dunbar,
Andreev's theorem on hyperbolic polyhedra,
Ann. Inst. Fourier (Grenoble) \textbf{57} (2007), no. 3, 825--882.
\url{https://doi.org/10.5802/aif.2279}.

\bibitem{SchlenkerHyperideal}
J.-M. Schlenker,
Hyperideal polyhedra in hyperbolic manifolds,
preprint, arXiv:math/0212355v2 (2003).
\url{https://arxiv.org/abs/math/0212355v2}.

\bibitem{Ushijima2006}
A. Ushijima,
A volume formula for generalised hyperbolic tetrahedra,
in: \emph{Non-Euclidean Geometries}, Mathematics and Its Applications, vol. 581,
Springer, Boston, MA, 2006, 249--265.
\url{https://doi.org/10.1007/0-387-29555-0_13}.

\bibitem{VesninEgorovUpper}
A. Yu. Vesnin, A. A. Egorov,
Upper bounds for volumes of generalized hyperbolic polyhedra and hyperbolic links,
Siberian Math. J. \textbf{65} (2024), no. 3, 534--551.
\url{https://doi.org/10.1134/S0037446624030042}.

\bibitem{EgorovP32}
A. Yu. Vesnin, A. A. Egorov,
The right-angled Coxeter group of minimal covolume in three-dimensional hyperbolic space,
Siberian Math. J. \textbf{66} (2025), no.~6, 1374--1389.
\url{https://doi.org/10.1134/S0037446625060047}.

\bibitem{VinbergGeometry}
E. B. Vinberg (ed.),
\emph{Geometry II}, Encyclopaedia of Mathematical Sciences, vol. 29,
Springer-Verlag, Berlin, 1993.
\url{https://doi.org/10.1007/978-3-662-02901-5}.

\end{thebibliography}
\end{document}